\documentclass[11pt,a4paper,reqno]{amsart}
 \usepackage{amsaddr,xcolor}
\usepackage[normalem]{ulem}
\usepackage{amsmath,amssymb,amsfonts,amscd,bbm}
\usepackage[mathscr]{euscript}
\usepackage[all]{xy}
\usepackage{color}

\usepackage{tikz-cd}

\theoremstyle{plain}
\newtheorem{theorem}{Theorem}[section]
\newtheorem{lemma}[theorem]{Lemma}
\newtheorem{proposition}[theorem]{Proposition}
\newtheorem{corollary}[theorem]{Corollary}
\theoremstyle{definition}
\newtheorem{definition}[theorem]{Definition}

\newtheorem{example}{Example}
\newtheorem{remark}[theorem]{Remark}
\newtheorem{remarks}[theorem]{Remarks}

\newcommand\<[1]{\langle\,#1\,\rangle}
\newcommand\norm[1]{\Vert #1 \Vert}

\newcommand{\mA}{\mathscr{A}}

\newcommand{\h}{\mathbb{H}}
\newcommand{\gm}{{\Gamma}}
\newcommand{\set}[2]{\{\,{\textstyle#1};\,{\textstyle #2}\,\}}

\newcommand{\mJ}{{\mathscr J}}\newcommand{\mI}{{\mathscr I}}
\newcommand{\mB}{{\mathscr B}}
\newcommand{\fL}{\mathfrak{L}}
\newcommand{\fR}{\mathfrak{R}}

\def\wasabi{\makeatletter %% access private macros
W\raisebox{.05 em}{e}\raisebox{-.05 em}{S}e\raisebox{.05 em}{B}ai
}
\newcommand{\bd}{{\ast\ast}}

\newcommand{\T}{\mathbb{T}}
\newcommand{\C}{\mathbb{C}}
\newcommand{\Z}{\mathbb{Z}}
\newcommand{\N}{\mathbb{N}}

\newcommand{\A}{\mathscr{A}}

\DeclareMathOperator{\supp}{supp}

\newcommand{\lpg}{\mbox{\tiny $L^p(\Gamma)$}}
\newcommand{\lqg}{\mbox{\tiny $L^q(\Gamma)$}}

\numberwithin{equation}{section}

\usepackage{amsmath,amscd,enumitem}

\newcommand\m[1]{\mathrm{m}_{\mbox{\tiny   #1}}}	

\newcommand\restr[2]{{% we make the whole thing an ordinary symbol
  \left.\kern-\nulldelimiterspace % automatically resize the bar with \right
  #1 % the function
  \vphantom{\big|} % pretend it's a little taller at normal size
  \right|_{#2} % this is the delimiter
  }}
\newcommand{\Cf}[1]{\ensuremath \mbox{\large${\mathbf{1}}$}_{#1}}
\newcommand\di[2]{d\left({#1},{#2}\right)}
\begin{document}
\newcommand{\I}{\mathbb{I}}

\title[Arens regularity of ideals]{Special subsets of discrete groups and  Arens regularity  of ideals in Fourier and group algebras}

%\author[Esmailvandi, Filali and Galindo]{R.  Esmailvandi, M. Filali \and  J. Galindo}

\author[Esmailvandi,  Filali, Galindo ]{ Reza Esmailvandi$^1$,  Mahmoud Filali$^2$,  Jorge Galindo$^3$}

\address{$^{1}$ Instituto Universitario de Matem\'aticas y Aplicaciones (IMAC),
Universitat Jaume I, E-12071, Castell\'on, Spain; \hfill\break \noindent E-mail: esmailva@uji.es}

\address{$^{2}$Department of Mathematical Sciences,
    University of Oulu, Oulu Finland;\hfill\break \noindent E-mail:  mfilali@cc.oulu.fi}

\address{$^{3}$ Instituto Universitario de Matem\'aticas y Aplicaciones (IMAC),
Universitat Jaume I, E-12071, Castell\'on, Spain;\hfill\break \noindent E-mail:  jgalindo@uji.es}

%\thanks{ Research of  the second named  author  supported by Universitat Jaume I, grant UJI-B2019-08}


\date{\today}

\begin{abstract}
Let $ \mA$ be a weakly sequentially complete Banach algebra containing a bounded  approximate identity that is an ideal in its second dual $\mA^\bd$, we call such an algebra a $\wasabi$ algebra.
 In the present paper we examine the Arens regularity properties of closed ideals of algebras in the \wasabi class. We observe that, although \wasabi algebras are always strongly Arens irregular, a variety of Arens regularity properties can be observed within their closed ideals.

    After characterizing   Arens regular ideals and strongly Arens irregular ideals, we proceed to  particularize to the main examples of \wasabi algebras, the convolution group algebras $L^1(G)$, $G$ compact, and the Fourier algebras $A(\Gamma)$, $\Gamma$ discrete and amenable. We find examples of Arens regular ideals in $L^1(G)$ and $A(\Gamma)$, both reflexive and nonreflexive and examples of strongly Arens irregular ideals that are not in the \wasabi class.      For this, we construct, in many noncommutative groups, a new class of Riesz sets which are not $\Lambda(p)$, for any $p>1$. Our approach also shows that every infinite Abelian group contains a Rosenthal set that is not $\Lambda(p)$, for any $p>0$. These latter results  could be of independent interest.
      \end{abstract}
\maketitle

\section{Introduction}
There are two  natural ways to define a multiplication in the  second dual  $\mA^\bd$ of a Banach algebra $\mA$ in such a way   that   the natural embedding of $\mA$ into $\mA^\bd$ is made into a Banach algebra homomorphism. This was observed,
 over seventy years ago, by R. Arens, \cite{A1} and \cite{A2}. Arens products, as these multiplications came to be known,  have since been extensively studied, see, e.g.,
\cite{civin-yood},\cite{Day},  \cite{pym}, \cite{Young} or, more recently,  \cite{balapy}, \cite{DL05}, \cite{DLS},
 \cite{filagali18}, \cite{filagali20}, \cite{filagali21}, \cite{filagali22}, \cite{FiSi}, \cite{ulger99} and the references therein.

 It was soon recognized that the interplay between both products, or, rather, how their properties compare, reflects some of the properties of the algebra. In most  prior research,  with the notable exception of   \cite{ulger99},  a fundamental  tool that consistently proves indispensable  is the utilization of  bounded (right or/and left) approximate identities. The reason for this can be glimpsed at  the beginning of Section \ref{Grosser}.

In this paper, we are primarily concerned  with  Arens products on  closed ideals of the  Banach algebras that  constitute  the $\ell^1$-side of Abstract Harmonic Analysis (as opposed to  the  $C^\ast$-side): group algebras and Fourier algebras.
%More specifally with  the topological and algebraic properties of their   Arens multiplications.
These ideals need not possess    bounded approximate identities   and, because of that, many conventional approaches to  Arens (ir-)regularity  do not apply to them. However, it is still possible to  use the bounded approximate identities of their ambient algebras to bring some of their structure back.
%$. With the notable exception of  the paper \cite{ulge99}, this is a little trodden approach .
% here that but  their  ideals may not contain an  (right or/and left) approximate identity,  even if these algebras normally have. This task was launched in  our previous paper \cite{EFG}, where closed ideals of
%the group algebra of a compact Abelian group were studied. We undertake here a  general approach that will encompass and extend  our previous results.
%we shall be concerned with Arens regularity and Arens irregularity of a class of Banach algebras which do not possess a bounded  (right or/and left) approximate identity.
%This is  a little trodden  approach,  the paper \cite{ulge99} being a notable exception.  {\tt Needs to be rewritten}

%\emph{Faute-de-mieux} we have chosen to name the algebras in this class,  \wasabi algebras. Closed ideals in a \wasabi algebra need not be in the class, {\tt but} nonetheless they share some of the characteristic traits of the class, but lack in general a bounded  (right or/and left) approximate identity.
%not all
%{\tt Wasabi is in Lemma 3.7, Proposition 4.3, Theorem 5.1, and Theorem 6.2}

Throughout the paper, we will say that a   Banach algebra $\mA$ is in the \wasabi class (or is a \wasabi  algebra) if $ \mA$ is weakly sequentially complete,   has a bounded approximate identity and  is an ideal in its second dual $\mA^{**}$.  If $\mA$ is in the \wasabi class and
  $pq$ and $p\diamondsuit q$, denote the    two Arens products of elements  $p, q\in \mA^\bd$ (we postpone  definitions to the next section), it is well known that  the translation maps
$q\mapsto pq$ and $q\mapsto q\lozenge p$,  are both weak$^\ast$-continuous only if $p\in \mA$. This is summarized by saying that \wasabi algebras are strongly Arens irregular (or sAir, for sort). This behaviour is not shared by  all the ideals of algebras in the \wasabi class. In fact  \"Ulger \cite{U} proved that in the second dual of  certain nonreflexive ideals of the group algebra of a compact Abelian group,  translations can even be continuous for every $p\in \mA^\bd$. That is, ideals of a \wasabi algebra  can  be Arens regular.

In this paper we  characterize  Arens regularity and  strong Arens irregularity of ideals of \wasabi algebras. We observe that, with few exceptions,   among the ideals of   such an algebra, the entire spectrum of behaviors associated with Arens irregularity can be found.  This task was launched in  our previous paper \cite{EFG}, where closed ideals of
the group algebra of a compact Abelian group were studied.

 Our last Section provides general constructions and examples that materialize these behaviours. We consider there the main examples of algebras in the \wasabi class: the algebra $L^1(G)$,  where $G$ is a compact,  not necessarily commutative, group and the Fourier algebra $A(\Gamma)$ of a discrete amenable group $\Gamma$. Improving, and generalizing to the Fourier algebra context,  constructions of suitable special sets known in  classic Harmonic Analysis, we  find ideals in $L^1(G)$ and $A(\Gamma)$ that are (1) infinite dimensional and reflexive (2)  Arens regular and not reflexive and (3) strongly Arens irregular, not necessarily in the
\wasabi class.
%
%  When $\mA$ is the group algebra of a compact Abelian group,  this result was  obtained originally by \"Ulger in \cite{U} 	and proved again by the authors in \cite{EFG}.
%
%
% \"Ulger asked in \cite{U} the interesting question
% whether Arens regular ideals in the group algebra of a compact Abelian group are necessarily Riesz ideals.
%	
%As already noted in Remark \ref{rem}, the equivalent condition (ii) in Theorem \ref{prop1} for the Arens regularity of $\mJ$ seems to be weaker that $\mJ$ being a Riesz ideal
%in $\mA.$ This points out towards a negative answer to \"Ulger's question, which we are still
%not able to provide, unfortunately.

\section{Notation}
Let $\mA$ be a  Banach algebra and let $\mA^*$ and $\mA^{**}$ be its first and second Banach duals, respectively. The multiplication of $\mA$ can be extended naturally to $\mA^{**}$ in two different ways. These multiplications arise
as particular cases of the abstract approach of R. Arens \cite{A1,A2} and can be formalized through the following three steps. For $a,b \in \mA$ ,  $\varphi\in \mA^*$
and $m,n $ in $\mA^{**}$; we define $\varphi\cdot a, a \cdot \varphi, m \cdot \varphi, \varphi\cdot m $ in $\mA^*$ and $m  n, m \diamondsuit n$ in $\mA^{**}$
as follows:
\begin{align*} \label{eq1}
\langle \varphi\cdot a , b \rangle &= \langle \varphi, a  b \rangle,
\qquad \qquad
\langle a \cdot \varphi  , b \rangle = \langle \varphi, ba \rangle\\
\langle m \cdot \varphi, a \rangle &= \langle m , \varphi\cdot a \rangle,
\qquad
\langle  \varphi\cdot m  , a \rangle = \langle m , a \cdot \varphi\rangle\\
\langle m n , \varphi\rangle &= \langle m , n \cdot \varphi\rangle,
\qquad
\langle  m \diamondsuit n  , \varphi\rangle = \langle n , \varphi\cdot m \rangle.
\end{align*}
When $mn=m\diamondsuit n$ for every $m,n\in \mA^{**}$, $\mA$ is said Arens regular.

We say that a bounded net $(e_\alpha)$ is a bounded right approximate identity (brai for short) in $\mA$,
when for every $a\in \mA,$ $\lim_\alpha ae_\alpha=a$ in norm.
Any weak$^*$-limit $e$ of $(e_\alpha)$ is a right identity in $\mA^{**}$. Bounded left approximate identities (blai) can be analogously defined, and any weak$^*$-limit $e$ of $(e_\alpha)$ is a left identity for $(\mA^{**}, \diamondsuit)$.
The
$(e_\alpha)$ will be a bounded approximate identity (a bai) if it is both brai and blai, and
%a bounded right approximate identity and a left  approximate identity.
%If $(e_\alpha)$ happens to be a bai
%bounded approximate identity (bai for short), then
$e$
is  called then a mixed identity  in $\mA^{**}$.
%We recall the definition of {\it mixed identity} which play a crucial role in our exposition. An element $e \in \mA^{**}$ is said to be a mixed identity if $e$ is a right identity for $(\mA^{**},  )$ and a left identity for $(\mA^{**}, \diamondsuit)$. By \cite[Proposition 2.9.16]{dal00}, an element  $e \in \mA^{**} $ is a mixed identity if and only if $e$ is a weak$^*$ cluster point of some bounded approximate identity (BAI)  in $\mA$. Let $\mA$ be a Banach algebra with BAI, we denote by $\mathcal E$ the set of the mixed identities in $\mA^{**}$.

%\remark
 %Let $\mA$ be a Banach algebra which is an ideal in its second dual. Then
%for each $m,n,p \in \mA^{**}$ and $\varphi\in \mA^*$, we have
%\begin{align}
 %   (m \cdot \phi) \cdot n= m \cdot (\varphi\cdot n).
%\end{align}
%Consequently,
%\begin{align}
 %   (n \diamondsuit p)   m= n \diamondsuit (p   m).
%\end{align}

\definition
Let $ \mA $ be a Banach algebra. Then the topological centres $\mathcal{Z}_1(\A^{**})  $ and $ \mathcal{Z}_2( \mA^{**}) $ of $ \mA^{**} $ are defined as:
\begin{align*}
\mathcal{Z}_1(\mA^{**})& =\lbrace m \in \mA^{**}: \text{the map}~ n\mapsto m   n ~ \text{is weak*-weak* continuous on }\mA^{**} \rbrace\\
&=\lbrace m \in \mA^{**}: m   n = m \diamondsuit n ~\text{for all }~ n \in \mA^{**} \rbrace \\
\mathcal{Z}_2( \mA^{**})
& =\lbrace m \in \mA^{**}: \text{the map}~ n\mapsto n \diamondsuit m ~ \text{is weak*-weak* continuous on }\mA^{**} \rbrace\\
&=\lbrace m \in \mA^{**}: n   m = n \diamondsuit m ~\text{for all }~ n \in \mA^{**} \rbrace.
\end{align*}
It is clear  that $\mA \subset\mathcal{Z}_1(\mA^{**}) \cap \mathcal{Z}_2( \mA^{**})$. The algebra $\mA$ is therefore Arens regular if and only if $$\mathcal{Z}_1(\mA^{**})=\mA^{**} =\mathcal{Z}_2( \mA^{**}).$$
When $\mA$ is commutative both topological centers coincide with the algebraic center
$\{m\in  \mA^\bd\colon mn=nm \mbox{ for all } n\in \mA^\bd\}$.
\definition
Let $ \mA $ be a Banach algebra. Then $ \mA $ is is called left strongly Arens irregular  (lsAir) if
$ \mathcal{Z}_1( \mA^{**}) = \mA $, right strongly Arens irregular  (rsAir) if
$ \mathcal{Z}_2( \mA^{**}) = \mA $, and strongly Arens irregular
(sAir) if
$$\mathcal{Z}_2( \mA^{**}) = \mA = \mathcal{Z}_1( \mA^{**}).$$

%\example
%Let $X=L^1(G)$ as a Banach space for an infinite locally compact group $G$, setting
%\begin{align*}
%    a \cdot x = a \ast x, \qquad x \cdot a =0 \qquad (a \in L^1(G), x \in X).
%\end{align*}
%Then $X$ is a Banach $L^1(G)$-bimodule, and the Banach space $\mA=L^1(G) \oplus_{\ell_1}X$ is a Banach algebra for the product defined by
%\begin{align*}
%    (a,x)(b,y)=(a \ast b , a \ast y).
%\end{align*}
%Note that $\mA^{**} = L^1(G)^{**}\oplus_{\ell_1}X^{**}$ and it can be checked that Arens multiplications satisfies
%\begin{align*}
%    (m , \Phi)   (n,\Psi)
%    &=(m   n, m   \Psi),\\
%    (m , \Phi) \diamondsuit (n,\Psi)
%    &=(m \diamondsuit n, \Phi \diamondsuit n  ),
%\end{align*}
%for $m,n \in L^1(G)^{**}$ and $\Phi,\Psi \in X^{**}$.
%It is now a consequence of \cite{gmm98} that
% $\mA$ is LsAir. Also one can obtain  $\mathcal{Z}_2(\mA)= L^1(G) \oplus X^{**}$. Therefore, $\mA$ is LsAir but it is not RsAir.
%
%
%

\remark
Let $ \mA $ be a Banach $*$-algebra. Then involution $ * $   on $ \mA $ extends to a linear involution on $\mA^{**}  $, and
$$(m   n)^* = n^* \diamondsuit m^*.$$
Thus, we have
\begin{align*}
   m \in \mathcal{Z}_1(\mA^{**}) \qquad \Longleftrightarrow \qquad  m^* \in \mathcal{Z}_2(\mA^{**}).
\end{align*}
Hence, for  a Banach $*$-algebra $ \mA $. lsAir, rsAir and sAir are equivalent.

In \cite{pym}, J. Pym considered the space $WAP(\mA)$ of {\it weakly almost periodic functionals} on $\mA$, this is the set of all $\varphi\in \mA^*$ such that the linear
map
\begin{align*}
    &\mA \longrightarrow \mA^* \\& a \mapsto  a \cdot \varphi\end{align*}
is weakly compact. A necessary and sufficient condition for a functional $\varphi\in \mA^\ast$ to  be in  $ WAP(\mA)$ is to  satisfy Grothendick's double limit criterion
\begin{align*}
    \lim_n \lim_m \langle \varphi, a_n b_m \rangle =  \lim_m \lim_n \langle \varphi, a_n b_m \rangle
\end{align*}
for any pair of bounded sequences $(a_n)_n$, $(b_m)_m$ in $\mA$ for which both the iterated limits exist. From this property, one may deduce that \[\langle m   n , \varphi\rangle = \langle m \diamondsuit n , \phi\rangle\mbox{
for every } m,n \in \mA^{**}\] if and only if $\varphi\in WAP(\mA)$.
So, $\mA$ is
Arens regular if and only if  $WAP(\mA) = \mA^{*}$, i.e., when the quotient $\mA^{*}/WAP(\mA)$
is trivial. This is the motivation for the following definition.

\definition
A Banach algebra $\mA$ is extremely non-Arens regular (enAr
for short) when $\mA^{*}/WAP(\mA)$ contains a closed subspace isomorphic to $\mA^{*}$.

%
%\remark \label{rem2}   (\cite [Lemma 2.1]{dul92}).
%Let $\mA$ be a Banach algebra with a BAI. If $\mA$ is an ideal in its second dual, then
%$$\mA^* \cdot \mA = WAP(\mA) = \mA \cdot \mA^*.$$

%\end{document}

\section{A decomposition of $\mJ^{**}$}\label{Grosser}
When $\mA$ has a brai, $\mA^\bd$ contains  right identities for the first Arens product.  If $e$ is such a right identity, multiplication by $e$  defines a projection on $\mA^\bd$ which produces the  decomposition $p=e  p+(p-e  p)$ for every
$p\in \mA^\bd$. Since $(p-e  p)\in (\mA^*\cdot \mA)^\perp$, we see that $\mA^\bd$ can be described as the internal direct sum
 \begin{equation}\label{desc:a2}\mA^{**}=e\mA^\bd \oplus (\mA^*\cdot \mA)^\perp,\end{equation}
 where $(\mA^*\cdot \mA)^\perp$ is easily identified with the subspace of right annihilators of $\mA^\bd$.
Recall that $\mA^*\cdot\mA=\{\varphi \cdot a:\varphi\in \mA^*,\;a\in \mA\}$ is a closed linear subspace of
$\mA^*$ by Cohen-Hewitt factorisation theorem (see for example \cite[Theorem 32.22]{hewiross2}).
It follows from this decomposition,  that $e\mA^\bd $ is isomorphic to the quotient $\mA^\bd/(\mA^*\cdot \mA)^\perp$. We have thus the following isomorphic decomposition of $\mA^{**}$
\begin{equation}\label{desc:a}\mA^{**}\cong (\mA^*\cdot \mA)^\ast \oplus (\mA^*\cdot \mA)^\perp.\end{equation}
 When  the brai $(e_\alpha)$ is contractive (i.e., of bound $1$) so that $\|e\|=1$,  the Banach algebras  $(\mA^*\cdot \mA)^*$ and $e\mA^{**}$ are not only isomorphic but also
isomorphically isometric.
%Furthermore,
%  is isometric to the
% Banach space direct sum
%If $e$ as any limit point of $(e_\alpha)$ in $\mA^{**}$, one can then observe that  $(\mA^*\mA)^\perp$ is complemented in  $\mA^{**}$, the complement can actually be identified with $e\mA^{**}$ which is then seen to be  isomorphic with  $(\mA^*\mA)^*$ as a Banach algebra. So  \eqref{desc:a} can be read as

If $\mA$ is a  right ideal in $\mA^{**}$ and $(e_\alpha)$ is a bai, the Banach spaces  $\mA^*\cdot\mA$, $\mA\cdot \mA^\ast$
and $WAP(\mA)$ are the same, see \cite[Theorem 1.5]{balapy}. So in this case, \eqref{desc:a} can be written as
 \begin{equation}\label{desc:awap}\mA^{**}\cong WAP(\mA)^*\oplus WAP(\mA)^\perp.\end{equation}

These results go back to  to Grosser \cite{Grosser}. For full details of this decomposition the reader is directed to \cite{balapy}.

Next we need
the multiplier algebra $M(\mA)$ of  $\mA,$  see for instance \cite{dal00}.
For this, let $B(\mA)$ be the Banach algebra of all bounded operators on $\mA$ and  recall that a $ L\in B(\mA)$ is a {\it left  multiplier} of $ \mA $ if $ L(ab) = L(a)b$ for $a,b \in \mA$,
$ R\in B(\mA)$ is a right multiplier of $ \mA $ if $R(ab) = aR(b)$  for $a,b \in \mA$. A
 multiplier  of $ \mA $ is a pair $ (L,R) $, where $ L $ is a left multiplier, $ R $ is a right multiplier, and
$aL(b)=R(a)b$  for $a,b \in \mA$.
Let $\mA$ still have a bai $(e_\alpha)$.
 Then \[M(\mA):= \{(L,R) : (L,R) ~\text{is a multiplier of } \mA\}\] is a closed, unital subalgebra of $ B(\mA) \oplus_\infty  B(\mA)^{op} $.

Furthermore,   $\mA^{**}$ is a Banach $M(\mA)$-bimodule, where the right action of $M(\mA)$ on $\mA^{**}$
is given by  $m\cdot (L,R)=R^{**}(m)$, here
$R^{**}:\mA^{**}\to \mA^{**}$ is the second adjoint of $R$.
So we may consider
the map  \begin{align}\label{eq13}
\fR_e:  M(\mA) \longrightarrow  \mA^{**},
\qquad
 \fR_e\colon (L,R)\mapsto  e.(L,R)=R^{**}(e).
\end{align}
If $(e_\alpha)$ is  contractive,   \cite[Theorem 2.9.49]{dal00} proves that this map
 is an  isometric embedding, Clearly,   $\fR_e(a)=R_a^{**}(e)=a$  for every $a \in \mA$.
Since,
 for each right multiplier $R$ of $\mA$,
\begin{equation}\label{onto}eR^{**}(e)=\lim_\alpha\lim_\beta e_\alpha R(e_\beta)=\lim_\alpha\lim_\beta R(e_\alpha e_\beta)
=\lim_\alpha R(e_\alpha)=R^{**}(e),\end{equation}
the $\fR_e$ map is  clearly into $ e\mA^{**} $.

When $ \mA $ in an ideal in $ \mA^{**} $, the map $ \fR_e$ is in fact onto $ e\mA^{**} $.
To see this, consider for  each $ m \in \mA^{**} $,  the left and right translate $L_m$ and $R_m$
of elements in $\mA$ by $m$.
Since $ \mA $ in an ideal in $ \mA^{**} $,  $ T_m = (L_m , R_m) \in {M(\mA)}, $ and so
 \[ \fR_e(L_m,R_m) = R_m^{**}(e) =e m,\] showing that  $ \fR_e $ is onto.
Therefore, when $ \mA $ in an ideal in $ \mA^{**} $ and  $(e_\alpha)$ is a contractive bai, $ M(\mA)$ is isometrically  isomorphic with  $e\mA^{**} $ (and accordingly with $(\mA\cdot\mA^*)^*$, $(\mA^*\cdot \mA)^*$ and $WAP(\mA)^*$) as a Banach algebra.

We may therefore consider the weak$^*$-topology $\sigma(M(\mA), WAP(\mA))$, we denote this topology by $w$.
So the ${w}$-convergence of a net $((L_\alpha,R_\alpha))_\alpha$ to an element $(L,R)$ in $M(\mA)$ means, by means of the identification given by (\ref{eq13}),
\begin{equation}\label{weak}\lim_\alpha\langle R_\alpha^{**}(e) , \varphi \rangle = \langle R^{**}(e) , \varphi \rangle \qquad (\varphi \in WAP(\mA)),\end{equation}
or equivalently, \[\lim_\alpha\langle L_\alpha^{**}(e) , \varphi \rangle = \langle L^{**}(e) , \varphi \rangle \qquad (\varphi \in WAP(\mA)).\]
The equivalence is due to the fact that, on $WAP(\mA)$, we have
\[L^{**}(e) = e \diamondsuit L^{**}(e) = R^{**}(e) \diamondsuit e = R^{**}(e)  e = R^{**}(e)\]
using  the facts that the first and second Arens products coincide on $WAP(\mA)$ for the third equality, $(L,R)$ is a multiplier for the second equality and $e$ is a mixed identity in $\mA^{**}$ for the rest.

Note now  that if $(L,R)\in M(\mA),$ and $(a_\alpha)$ is a net in $\mA$ with $R^{**}(e)$ as its weak$^*$-limit in $\mA^{**}$, then
\[\lim_\alpha\langle R_{a_\alpha}^{**}(e) , \varphi \rangle = \lim_\alpha\langle a_\alpha , \varphi \rangle=\langle R^{**}(e) , \varphi \rangle \quad\text{for every}\quad\varphi \in WAP(\mA)\]
means that that $M(\mA)=\overline{\mA}^w$.

The relation between $e\mA^\bd $, $M(\mA)$ and $WAP(\mA)$ is pictured in the following diagram
\begin{center}
\begin{tikzcd}[column  sep=9em]
 e  \mA^{**}
\arrow[d] &
\\
 WAP(\mA)^*
\arrow[u]
\arrow[d]
&M(\mA)
\arrow[ul, bend right=20, " \fR_e"]
\arrow[dl,bend left=20, " \fL_e"]
\\ \mA^{**} \diamondsuit e
\arrow[u] &
\end{tikzcd}
\end{center}

Let $\mA$ be still an ideal in $\mA^{**}$ with a  bai
 $(e_\alpha)$ and let $\mJ$ be a
closed  ideal in $\mA$.
For our study of how Arens products behave in $\mJ^{**}$, we shall need  a decomposition similar to that of \ref{desc:a2} for
$\mA^{**}$.
The arguments  used to prove this decomposition cannot be applied directly since in most of the interesting cases studied in the sequel,
$\mJ$ has no brai,
and so there is no right identity in $\mJ^{**}$.
But, as we see next, the use of a bai in $\mA$ helps to recapture the decomposition for $\mJ^{**}$ with the same properties as previously for $\mA^{**}$.  Note that in the case $\mJ=\mA$, this theorem  just brings back the discussion of the precedent paragraphs.

 Let $i:\mJ\to\mA$ be the inclusion map and $i^*$ be its adjoint, i.e.,
$i^*: \mA^*\to \mJ^*$ is the restriction map.

\begin{theorem} \label{8}
Suppose that  $ \mA $ is a  Banach algebra, which is an ideal in its second dual $ \mA^{**}  $,   and has a  bai $ (e_\alpha)_\alpha $. Let $ e $ be a mixed identity of $ \mA^{**} $ associated with $ (e_\alpha)_\alpha $.
Let $ \mJ $ be a  closed ideal of $ \mA $ and $\overline{\mJ}^w $ be its $w$-closure in the multiplier algebra $M(\mA).$ Then the identities \[m=em+(m-em)\;\text{and}\; m=m\diamondsuit e+(m-m\diamondsuit e)\quad, m\in \mJ^\bd,\] give  the decompositions\begin{align}
  \label{desc:j}
\mJ^{**}&=e\mJ^{**}  \oplus (\mJ^*\cdot\mA)^\perp =  e\mJ^{**}  \oplus i^*(WAP(\mA))^\perp,\\
\mJ^{**}&=  \mJ^{**} \diamondsuit e \oplus (\mA\cdot\mJ^* )^\perp=    \mJ^{**} \diamondsuit e   \oplus i^*(WAP(\mA))^\perp.\notag
\end{align}
These decompositions have the following properties.
\begin{enumerate}
\item If $(e_\alpha)_\alpha$ is contractive, $\overline{\mJ}^w $ is a closed ideal of $M(\mA)$ and $ e\mJ^{**}\simeq \overline{ \mJ}^w \simeq\mJ^{**}\diamondsuit e$.
%\item $\fL_e(\wJ) = \mJ^{**}\diamondsuit e$,
%\item $i^*(WAP(\mA))^\perp=(1-e)\mJ^{**}=\mJ^{**}\diamondsuit (1-e),$
\item $\mJ^{**}\,i^*(WAP(\mA))^\perp=i^*(WAP(\mA))^\perp\diamondsuit \mJ^{**}=\{0\}.$
%\item $e(mn)=(m\diamondsuit n)\diamondsuit e$ when $\mA$ is commutative ({\tt I cannot prove it otherwise.}
%\item  $ emn=(em)(en)=(em)\diamondsuit(en)$ for every $n,m\in \mJ^{**}.$ {Lemma \ref{reg}!}
\item $i^*(WAP(\mA))^\perp$ is a closed ideal of $\mJ^{**}.$
\end{enumerate}
\end{theorem}

\begin{proof}
One can proceed, as  with  the decomposition of $\mA^{**}$,
fixing  a right identity $e\in \mA^{**}$,
then split $\mJ^{**}$ as
\begin{equation}\label{direct}\mJ^{**}=e\mJ^{**}\oplus (1-e)\mJ^{**}.\end{equation}
Using the fact that $\mJ$ is an ideal in $\mA,$ this is in fact  the same as \[\mJ^{**}=\left(e\mA^{**}\oplus (1-e)\mA^{**}\right)\cap \mJ^{**}.\]

Let $\overline{\mJ}^w$ be the ${w}$-closure of $\mJ$ in $M(\mA)$.%
%Using this fact and the identity $R^{**}(m)=mR^{**}(e)$ on $WAP(\mA)$, it follows that $\overline\mJ^w$ is a closed ideal of $M(\mA).$
%Another quicker way for deducing that $\overline\mJ^w$ is an ideal of $M(\mA)=\overline{\mA}^w$  is to pass by the  isometric isomorphisms between $M(\mA)$ and $e\mA^{**}$ as already checked,
%and between $\overline{\mJ}^w$ and  $e\mJ^{**}$ as we prove next.
% For this,
%We check that  $\overline{\mJ}^w$ is isometrically isomorphic to $e\mJ^{**}$.
%Note that  $\mA^{**}$ is a Banach $M(\mA)$-bimodule, the right action of $M(\mA)$ on $\mA^{**}$
%being $m.(L,R)=R^{**}(m)$ where
%$R^{**}:\mA^{**}\to \mA^{**}$ is the second adjoint of $R$.

Consider the restriction  of the map defined in \eqref{eq13} to $\overline{\mJ}^w$, denote it also by $\fR_e$. So
\[\fR_e:\overline{\mJ}^w\to \mA^{**},\quad \fR_e(L,R)=R^{**}(e)=e.(L,R).\]
 As already mentioned, when $(e_\alpha)_\alpha$ is contractive,  \cite[Theorem 2.9.49]{dal00} shows that this map
 is an  isometric embedding. The same can be said of the restriction of $\fL_e$ to $\overline{\mJ}^w$,
\[\fL_e:\overline{\mJ}^w\to \mA^{**},\quad \fL_e(L,R)=L^{**}(e)=(L,R)\cdot e.\]
% $\fR_e$ is a a linear isometry is clear and follows precisely as in the case of $\mA$ (see \cite[page 328]{Dales}. Since
%\begin{equation}\label{onto}eR^{**}(e)=\lim_\alpha\lim_\beta e_\alpha %R(e_\beta)=\lim_\alpha\lim_\beta R(e_\alpha e_\beta)
%=\lim_\alpha R(e_\alpha)=R^{**}(e),\end{equation}
%it is also clear that  $\fR_e$ is a homomorphism. Therefore,  $\fR_e$ is an isometric embedding.
Assertion (i)  will then be proved as soon as we check that $\fR_e$ maps $\overline{\mJ}^w$ onto  $e\mJ^{**}$ and $\fL_e$ maps $\overline{\mJ}^w$ onto  $\mJ^{**}\diamondsuit e $.

We first show that the range of $\fR_e$ is contained in $e\mJ^{**}$.
Let $(L,R)\in \overline{\mJ}^w$ and $(a_\alpha)$ be  a net in $\mJ$ which converges
to $(L,R)$ in $M(\mA)$ in the $w$- topology, i.e., $\langle R_{a_\alpha}^{**}(e), \varphi\rangle$ converges to $\langle R^{**}(e), \varphi\rangle$ for every $\varphi\in WAP(\mA)$.
Then,  for every $a\in \mA$ and $\varphi\in\mA^*,$
%\lim_\alpha\langle (L_{a_\alpha}, R_{a_\alpha}\cdot a,\varphi\rangle&=
%\lim_\alpha\langle (L_{a_\alpha}(a),\varphi\rangle\\&=
we find  \begin{align*} \lim_\alpha\langle {a_\alpha}a,\varphi\rangle
&=\lim_\alpha\langle {a_\alpha},a\cdot \varphi\rangle
\lim_\alpha\langle {e a_\alpha},a\cdot \varphi\rangle
=\lim_\alpha\langle R_{a_\alpha}^{**}(e),a\cdot \varphi\rangle\\&=\langle R^{**}(e),a\cdot \varphi\rangle=
\langle R^{**}(e)a,\varphi\rangle=\langle R^{**}(a),\varphi\rangle=\langle R(a),\varphi\rangle.\end{align*}
Since $(a_\alpha a)$ is a net in $\mJ,$ this means that $R(a)$ is in the weak closure (i.e, ($\sigma(\mA, \mA^*)$-closure) of $\mJ$,
thus in $\mJ. $

In particular, $R(e_\alpha)\in \mJ$ for every $\alpha$ whenever $(L,R)\in \overline{\mJ}^w$.
%$\fR_e(\mu)=R^{**}(e)=\lim_\alpha R(e_\alpha), $
%where \[\langle R(a), \varphi\rangle=\langle\mu\ast a, \varphi\rangle=\langle\mu,  a %\varphi\rangle =\lim_\alpha\langle a_\alpha, a\varphi\rangle=\lim_\alpha\langle a_\alpha %a,\varphi\rangle,\] for each $a\in \mA,$

Accordingly,  \[\fR_e(L,R)=R^{**}(e)=\lim_\alpha R(e_\alpha)\in \mJ^{**}\quad\text{for every}\quad(L,R)\in \overline{\mJ}^w.\]
Since by (\ref{onto}), $R^{**}(e)=eR^{**}(e),$ this shows that the map $\fR_e$ maps $\overline{\mJ}^w$ into $e\mJ^{**}$.

Finally, we check that $\fR_e$ maps $\overline{\mJ}^w$  onto $e\mJ^\bd$.
So
let $m \in \mJ^{**} $ and $(a_\beta)_\beta$ be a bounded net in $\mJ$ such that $\lim_\beta a_\beta = m$ in the weak$^*$-topology.  After passing to a subnet if necessary, we may assume that the net  $(a_\beta)_\beta$ converges in the $w$- topology of $M(\mA)$ to some element $(L,R)$ in $\overline{\mJ}^w$. Then for every $\varphi \in \mA^*$, we find
\begin{align*}
  \langle em  , \varphi  \rangle
  &=\lim_\alpha \langle e_\alpha m, \varphi  \rangle
  =\lim_\alpha \langle m,  \varphi\cdot e_\alpha   \rangle\\
  &=\lim_\alpha \lim_\beta \langle a_\beta , \varphi\cdot e_\alpha   \rangle
  =\lim_\alpha \lim_\beta \langle R_{a_\beta}^{**}(e) ,  \varphi\cdot e_\alpha   \rangle\\
  &=\lim_\alpha  \langle R^{**}(e) ,   \varphi \cdot e_\alpha\rangle
  =  \langle eR^{**}(e)   , \varphi  \rangle\\
  &=  \langle R^{**}(e) ,   \varphi  \rangle,
\end{align*}
  where,  for the fifth equality, we use \eqref{weak}.
Hence $em =   R^{**}(e)$, as wanted.

%The proof for the left action is identical, so
 Similarly, with the left action of $M(\mA)$ on $\mA^{**}$ given by $(L,R)\cdot m=L^{**}(m)$ and  the map
\[\fL_e:\overline{\mJ}^w\to \mA^{**},\quad \fL_e(L,R)=L^{**}(e)=(L,R)\cdot e,\]
  $\mJ^{**}\diamond e$ is isometrically isomorphic to $\overline{\mJ}^w$.

Statement   (i) has now been proved.

As for the second summand $(1-e)\mJ^{**}$ in (\ref{direct}), note that
since $\mJ$ is an ideal of $\mA,$ the actions  of $\mA$ on $\mA^*$
(as defined at the beginning of Section 2) may be restricted to $\mJ^*$, and so the following identities hold for every $a\in \mA$ and $\varphi\in \mA^*:$
%\begin{equation}
\begin{align}\label{WAPJ}i^*(a\cdot\varphi)&=a\cdot i^*(\varphi)\quad\text{and}\quad i^*(\varphi\cdot a)=i^*(\varphi)\cdot a\quad\text{and so}\\&i^*(WAP(\mA))=\mA\cdot\mJ^*=\mJ^*\cdot\mA.\end{align}
%\end{equation}
Let now
\begin{align*}i^*(WAP(\mA))^\perp&=\{r\in \mJ^{**}:\langle r,i^*(WAP(\mA))\rangle=\{0\}\}
\\&=\{r\in \mJ^{**}:\langle r,WAP(\mA)\rangle=\{0\}\},\end{align*}
where in the second equality we regard $r$ as element in $\mA^{**}$, i.e., identifying $r$ with $i^{**}(r).$

%Note that  if $r\in i^*(WAP(\mA))^\perp$ and $\varphi\in \mJ^*,$ then $r\cdot\varphi=0$
%since
% \[\langle r\cdot \varphi, a\rangle=\langle r, \varphi\cdot a\rangle=0\quad\text{for every}\quad a\in \mA.\]
%Therefore, for every $m\in\mJ^{**}$, we have
%\[\langle mr, \varphi\rangle=\langle m, r\cdot \varphi\rangle=0,\] and so the elements in
So regarding the element in $i^*(WAP(\mA))^\perp$ as elements in $\mA^{**},$ we see that they are right annihilators of $\mA^{**}$.
In particular, $er=0$ and so $r\in (1-e)\mJ^{**}$. Thus, $i^*(WAP(\mA))^\perp\subseteq (1-e)\mJ^{**}.$

Conversely, it is also  clear that $(1-e)\mJ^{**}\subseteq i^*(WAP(\mA))^\perp$ since $em= m$ on $WAP(\mA)$ for every $m\in\mA^{**}$.
\medskip

Moreover, for $r\in i^*(WAP(\mA))^\perp$ and $m\in \mJ^{**}$, $rm\in i^*(WAP(\mA))^\perp$
since, as known, $WAP(\mA)$ is introverted (see for example \cite[Corollary 5.8]{DL05}.
Therefore $ i^*(WAP(\mA))^\perp$ is an ideal in $\mJ^{**}$.
 In the same manner, one can check the analog statement with the second Arens product.
This explains the form of the decompositions in \eqref{desc:j} and proves the remaining  statements in the theorem.
\end{proof}
 In specific cases, such as  the group algebra of a compact group or the Fourier algebra of an amenable discrete group, the  decomposition outlined in Theorem \ref{8} can be made more concrete, see Corollary \ref{decs:concr}.

\section{Some special sets in $\widehat{G}$  corresponding ideals in $\mA$.}
\label{special}
%defined in locally compact groups
The group algebra $L^1(G)$ of a compact group is the Banach algebra on which the \wasabi class is modelled. Its weak sequential completeness and the existence  of contractive bais are well-established facts. That $L^1(G)$ is an ideal in $L^1(G)^\bd$ when $G$ is a compact group (and only then) seems to have been first proved in \cite{wong71}.
%  If we denote the Haar measure of a compact group by $\mg$, the group algebra $L^1(G)$ is made of those functions $f\colon G\to \C$ that are absolutely integrable with respect to the Haar measure of $G$ with norm $\norm{f}_1=\int |f(x)| \dmg(x)$.

In order to introduce some of the concepts that will be used to identify Arens regularity properties in ideals of \wasabi algebras, we start looking at $L^1(G)$.   We first  recall how ideals of $L^1(G)$, $G$ a compact group,  are determined by subsets of its dual object $\widehat{G}$.

Consider  a compact group $G$. Its dual object
 $\widehat{G}$ is the set of  equivalence classes of irreducible unitary representations of $G$.  If  $\mu\in M(G)$ is a bounded measure,  we denote by $\widehat{\mu}$ its Fourier-Stieltjes transform.
 This is a function that sends every $\pi \in \widehat{G}$ to an operator $\widehat{\mu}(\pi)$ on the Hilbert space $\h_\pi$ on which
	$\pi$ operates 	($\C$, if $G$ is Abelian). Recall that the correspondence $\mu \mapsto \widehat{\mu}$ transforms convolutions into pointwise products. We refer to Section 28 of \cite{hewiross2} for  a complete description of the properties of the Fourier-Stieltjes transform.

 If  $ \mathcal{X} $ is a linear subspace of $M(G) $ and  $ E \subset \widehat G $, we denote by $ \mathcal{X}_E  $   the  subspace of $ \mathcal{X} $, given by
\[\mathcal{X}_E   = \lbrace  \mu \in \mathcal{X}:  \widehat{\mu}(\gamma) =0 ~ \text {for}~ \gamma \in \widehat{G}\setminus E \rbrace.\]
If $\mathcal{X}$ is a subalgebra of $M(G)$, then $\mathcal{X}_E$ is a closed ideal in $\mathcal{X}$.
It is well known that every closed ideal of $L^1(G)$ can be represented as $L^1_E(G)$ for some  $E \subset \widehat G$, see \cite[Theorem 38.7]{hewiross2}.

It should be possible, therefore, to codify any given  property  of a closed ideal of $L^1(G)$  in terms of some  property of the corresponding  subset of  $\widehat{G}$.  In the next paragraphs we will  introduce the  property  of a subset  of $\widehat{G}$ which does this job for  Arens  regularity. This property, and another one closely   related, will be then extended   to ideals of general \wasabi algebras. ideas  will be exploited in the upcoming section.
 In the final section of the paper, we will address other classes of subsets of $\widehat{G}$ that give rise to ideals possessing  interesting  properties from the  point of view of Arens-regularity. They will turn out to be classes  of subsets with special  combinatorial, analytic or arithmetic properties that have made them an object of study in Harmonic Analysis for decades.

 A subset  $ E $ of $ \widehat G $ is said to be a  {\it Riesz set}, when every measure $\mu\in M_E(G)$ is absolutely continuous, i.e., when $M_E(G)=L_E^1(G)$. The origin of the term is in the
  F. and M. Riesz theorem  (see, e.g. \cite[Theorem 8.2.1]{rudin62})  which proves  that   $\N$ is a Riesz set in $L^1(\T)$.

We say that a  subset  $ E $ of $ \widehat G $ is a {\it  small-1-1  set}  if $M_E(G)*M_E(G)\subseteq L^1_E(G) $.
The notion of  small-1-1 set  is apparently stronger than the better known notion  of small-2 set   which requires that $\mu\ast  \mu \in L^1_E(G)$ for every $\mu\in  M_E(G).$ When $G$ is Abelian,  every small-2 set is small-1-1, simply because $2\mu_1\ast \mu_2=(\mu_1+\mu_2)^2-\mu_1^2-\mu_2^2$. It
is still unknown to the authors whether small-1-1 sets (and so
small-2 sets) are Riesz. As already mentioned by \"Ulger in \cite[page 273]{U}, this is a long standing open problem that goes back to
Glicksberg \cite{Glick}.

All the previous definitions extend  seamlessly  to the Fourier algebra of a discrete group $\Gamma$.
For the definition and basic properties of the Fourier algebra $A(\Gamma)$, and the closely related Fourier-Stieltjes algebra $B(\Gamma)$, we refer to the recent monograph \cite{kanilaubook}. It is enough for the present discussion to say that $B(\Gamma)$ is a Banach algebra whose members are bounded complex-valued functions on $\Gamma$, matrix coefficients of unitary  representations of $\Gamma$ to be more precise, that $A(\Gamma)$ is an ideal of $B(\Gamma)$ and that   the dual Banach space of $A(\Gamma)$ is isometrically isomorphic to the von Neumann algebra $VN(\Gamma)$ generated by the convolution operators on $\ell^2(\Gamma)$, $f\mapsto h\ast f$, $h\in \ell^1(\Gamma)$. The Fourier algebra $A(\Gamma)$ is always weakly sequentially complete, as every predual of a von Neumann algebra, and it is an ideal in $A(\Gamma)^\bd$ \cite{lau81}. When $\Gamma$ is amenable, $A(\Gamma)$ has a contractive bai, see \cite[Theorem 2.7.2]{kanilaubook}.  $A(\Gamma)$ is therefore a \wasabi algebra for every discrete amenable group. When $\Gamma$ is a commutative group with character group $G$, the Fourier and Fourier-Stieltjes transforms establish linear isometries between $A(\Gamma)$ and  $L^1(G)$, and between $B(\Gamma)$ and $M(G)$.

If $\Gamma$ is a discrete group, $ \mathcal{X} $ is a linear subspace of $B(\Gamma) $  and  $ E \subset  \Gamma $, we denote by $ \mathcal{X}_E  $   the  subspace of $ \mathcal{X} $, given by
\[\mathcal{X}_E   = \lbrace  \varphi\in \mathcal{X}: \varphi(\gamma) =0 ~ \text {for}~ \gamma \in \Gamma \setminus E \rbrace.\]
If $\mathcal{X}$ is a subalgebra of $B(\Gamma)$, then $\mathcal{X}_E$ is a closed ideal in $\mathcal{X}$.
%he concept of Riesz set can be naturally extended to Fourier algebras of discrete groups, in that case
A subset $E\subseteq \Gamma$ is then said to be a \emph{Riesz set} if $A_E(\Gamma)=B_E(\Gamma)$, i.e., if every function of $B(\Gamma)$ supported on $E$ is actually in $A(\Gamma)$.

%\item {\it small-2 set } if $ \mu \ast \mu \in L^1_E(G)  $ for every $\mu \in M_E(G) $.
%
%
%\item   {\it t-set} if there exists a finite subset $ F_E $ of $ \widehat{G} $ containing $ 0  $ such that $ E \cap (E-\gamma) $ is finite, for each $ \gamma \notin F_E $.

%\item {\it small $L_E^1(G)^\bd$-2 set } if $ \{ p^2 : p \in L^1_E(G)^\bd\}\subseteq L^1_E(G). $

We now extend the concept of Riesz set and small-1-1 set to  general \wasabi algebras.

\begin{definition}\label{def1}
Let $ \mA $ be a Banach algebra with a  bai $ (e_\alpha)_\alpha $ and an associated
 mixed identity $e$ in $ \mA^{**} $.
We say that a closed ideal $\mJ$ of $\mA  $ is a
\begin{enumerate}
\item {\it Riesz ideal} if  $e\mJ^{**}=\mJ^{**}\diamondsuit e=\mJ$. If $(e_\alpha)_\alpha$ is contractive, this is the same as   $\overline \mJ^w  = \mJ$.
%\item[3.] {\it small-2 ideal} if  $ T^{2} \in  J$ for every $ T \in  \wJ$.
%\item[4.] {\it $\Lambda(1)$-ideal} if $  J$ is reflexive.
%\item[5.] {\it hypercoset ideal} if $\wJ$ is unital.
%\item[6.] {\it co-Rosenthal} if $J^\perp \subseteq WAP(\mA)$.
\item{\it small-1-1 ideal} if  $e\mJ^{**}\mJ^{**}\subseteq \mJ$,
i.e., for contractive $(e_\alpha)_\alpha$,  if $\overline \mJ^w \overline \mJ^w  \subseteq \mJ$.
\end{enumerate}
%{\tt I took away {\it hypercoset ideal}}
\end{definition}

\begin{remarks}\label{New} ~
\begin{enumerate}
\item
 { If $G$ is a compact group and $E\subseteq \widehat{G}$, then $L_E^1(G)$ is a Riesz ideal (respectively, a small-1-1 set) precisely when $E$ is a Riesz set (respectively, a small-1-1 set).}
 \item {By Theorem \ref{8}(iii), either condition $e\mJ^{**}=\mJ$ or $\mJ^{**}\diamondsuit e= \mJ$ is enough to define a Riesz ideal. This can also be verified directly as follows. Putting $m=em+r$ with $r\in WAP(\mA)^\perp$, we see that for $a\in\mJ,$
\[em=a\Longrightarrow m\diamondsuit e=(em+r)\diamondsuit e=(a+r)\diamondsuit e=a\diamondsuit e=a.\] Similarly, $m\diamondsuit e=a\Longrightarrow em=a.$}
\medskip
\item As in the previous remark, either condition   $\mJ^\bd\diamondsuit\mJ^\bd \diamondsuit e \subseteq \mJ$ or $e\mJ^\bd\mJ^\bd \subseteq \mJ$
is enough to define a small-1-1 ideal. For, using three times that  $pq-p\diamondsuit q\in WAP(\mA)^\perp$ for every $p,\,q\in \mA^\bd$, one can deduce that
        $(m\diamondsuit n)\diamondsuit e-emn$ is an annihilator, for every $m,\,n\in \mA^\bd$ and, hence, that $emn\in \mJ$ if and only if $(m\diamondsuit n)\diamondsuit e\in \mJ$.
\item
{ A closed ideal  $\mI$ of a  Riesz ideal $\mJ $ in $\mA$ is also  a Riesz ideal in $\mA$.}
 This is immediate from the fact that $e\mI^{**}\subseteq \mI^{**}\cap \mJ\subseteq \mI$ since $\mI$ is an ideal in $\mA$ and $\mJ$ is a Riesz ideal, the second inclusion is due to the convexity of $\mI$ as the weak limit of any net in $\mI$ stays in $\mI.$
%To see this, note first that $e\mI^{**}\subseteq \mJ$ since $\mJ$ is a Riesz ideal. So for every $u\in \mI^{**}$, we may see $eu$ as %an element  $v\in \mJ.$
%Let then  $(a_\alpha)$ be a net in $\mI$ with $u$ as its weak$^*$-limit. Then for every $a\in \mA$ and $\varphi\in \mA^*,$
%we have \[\lim_\alpha\langle (ea_\alpha)a,\varphi\rangle= \lim_\alpha\langle (ea_\alpha),a\cdot\varphi\rangle= \langle %eu,a\cdot\varphi\rangle=
%\langle v, a\cdot\varphi\rangle=\langle va, \varphi\rangle,\] where the second equality is due to $a\cdot\varphi\in WAP(\mA).$
%In other words,
%the net  $((ea_\alpha)a)_\alpha$ which is in $\mI$ has $va$ which is in $\mJ$ as its weak limit.
%Since $\mI$ is convex, the weak limit $va$ must be in $\mI$ for every $a\in \mA$.
%In particular, $(ve_\alpha)_\alpha$ is a net in $\mI$ when $(e_\alpha)$ is a bai in $\mA$, and so is $v$.
% Consequently, $eu=v\in \mI$, showing that $\mI$ is a Riesz ideal too.
\item
{ The same argument shows that  a closed ideal  of a  small-1-1 ideal is also  a small-1-1 ideal}.
%\item { It follows from its definition that for a Riesz ideal $\mJ$,
%\[\mJ^{**}\cdot \mJ^*=e\mJ^{**}\cdot \mJ^*\subseteq \overline{\mA\cdot \mJ^*} = \overline{i^*(WAP(\mA))} \quad\text{and}\]
%\[ \mJ^*\cdot \mJ^{**}= \mJ^*\cdot (\mJ^{**}\diamondsuit e)\subseteq\overline{\mJ^*\cdot \mA} =\overline{i^*(WAP(\mA))}.\]
 \end{enumerate}
%\quad\text{i.e.,}\\ \mJ^{**}\cdot \varphi&=e\mJ^{**}\cdot \varphi=\mJ\cdot \varphi\;\text{for every}\;\;\varphi\in \mJ^*, %\text{equivalently,}\\ \varphi\cdot\mJ^{**}&= \varphi \cdot (\mJ^{**}\diamondsuit e)
%=\varphi\cdot\mJ\;\text{for every}\;\;\varphi\in \mJ^*.\end{align*}}
\end{remarks}

To motivate our next definition, we need the following lemma.

\begin{lemma} In a Banach algebra with a   bai, a closed ideal $\mJ$ is a Riesz ideal if and only if for every $m\in \mJ^{**}$ there exists $a\in \mA$ (which is a fortiori in $\mJ$ by the fact that $\mJ^{**}\cap \mA=\mJ$)
such that \begin{align}\label{Riesz}m\cdot\varphi=(em)\cdot\varphi=a\cdot \varphi\quad\text{ for every}\; \varphi\in \mJ^*,\; \end{align}
\begin{align*} \text{equivalently,}\;\varphi\cdot m=\varphi\cdot (m\diamondsuit e)=\varphi\cdot a\text{ for every}\;\;\varphi\in \mJ^*.
\end{align*}
\end{lemma}

\begin{proof} If $\mJ$ is a Riesz ideal in $\mA$, then it is clear using the decomposition (\ref{desc:j}) that
for every $m\in \mJ^{**}$ there exists $a\in \mJ$
such that \[m\cdot\varphi=(em)\cdot\varphi=a\cdot \varphi\quad
\quad
\text{for every}\quad \varphi\in \mJ^*.\]
For the converse, suppose that  for every $m\in\mJ^{**}$ there exists $a\in \mJ$ such that $m\cdot\varphi=a\cdot \varphi$, and observe that
 \begin{align*}\langle em,\varphi\rangle=\langle e,m\cdot\varphi\rangle=\langle e,a\cdot\varphi\rangle&=
\langle a,\varphi\rangle\;
%\\ \langle m\diamondsuit e,\varphi\rangle=\langle e,\varphi\cdot m\rangle=\langle e,\varphi\cdot %a\rangle&=\langle a,\varphi\rangle,\;
\text{for}\; m\in \mJ^{**},\; a\in \mJ\;\text{ and}\;
 \varphi\in \mJ^*.\end{align*} Thus $em=a$, and so     the converse is also true.
\end{proof}

In view of the previous Lemma, a weakening of  the condition in (\ref{Riesz}) to $m\cdot\varphi\in \overline{\mA\cdot\mJ^{*}}=\overline{i^*(WAP(\mA))}$ (recall \eqref{WAPJ})  results in a weakening of the definition of Riesz ideals in $\mA$. This is what we do in the next definition.

\begin{definition}\label{def2}
Let $\mA$ be a Banach algebra with a  bai.
A closed ideal $\mJ$ of $\mA  $ is a {\it weak Riesz ideal} if for every $\epsilon>0,$ $m\in \mJ^{**}$ and $\varphi\in \mJ^*$, there exist $a, b\in \mA$ and $\psi, \theta\in \mJ^*$
such that \[\|m\cdot\varphi-a\cdot \psi\|_{\mJ^*}<\epsilon\quad\text{and}\quad\|\varphi\cdot m- \theta\cdot b\|_{\mJ^*}<\epsilon,\;i.e.,\]
\[\mJ^{**}\cdot\mJ^*\cup \mJ^*\cdot\mJ^{**}\subseteq\overline{i^*(WAP(\mA))}.\]
\end{definition}

Although weak Riesz ideals seem to be formally closer to being Riesz ideals than small-1-1 ideals, we  still do not know if Riesz and weak Riesz ideals are the same.

We observe next that weak Riesz ideals in a \wasabi algebra (and so, by Theorem \ref{prop1} below,  Arens regular ideals) are necessarily small-1-1.
We first need  a consequence of \cite[Lemma 3.1]{balapy}.

%\[\forall m\in\mJ^{**}\; \exists a\in\mA \mbox{ such that } m\cdot\varphi=(em)\cdot\varphi=a\cdot \varphi\quad\text{ for every}\; %\varphi\in \mJ^*\]
%\[\forall m\in\mJ^{**}\; \forall\varphi\in
%\mJ^*\; \exists a\in\mA \mbox{ such that } m\cdot\varphi=(em)\cdot\varphi=a\cdot \varphi \]
%\[\forall m\in\mJ^{**}\; \forall\varphi\in
%\mJ^*\;\exists a\in\mA\; \exists\psi\in\mJ^* \mbox{ such that } m\cdot\varphi=(em)\cdot\varphi=a\cdot \psi\]

\begin{lemma}  \label{lem5}
Let $ \mA $ be a Banach algebra in the \wasabi class. Let $ m \in \mA^{**}$  satisfy $\mA^* \cdot m  \subseteq WAP(\mA)$ or  $m \cdot \mA^* \subseteq WAP(\mA)$. Then $m =a_0 +r$ for some $a_0\in \mA$ and $r\in WAP(\mA)^\perp$.
\end{lemma}

\begin{proof} For  $ m \in \mA^{**} $, let $L_m\colon  \mA\to \mA$ and $R_m\colon \mA \to \mA$ be the left and right translate, respectively,
of elements in $\mA$ by $m$. Recall that, if $m \in \mA^{**}, \varphi \in \mA^*  $, and $a \in \mA$,  we have $L^*_m(\varphi)=  \varphi \cdot m $ and $R^*_m(\varphi)= m \cdot \varphi$. Our assumption
$m \cdot \mA^*  \subseteq WAP(\mA)$ means then that
$R_m^*$ maps $\mA^*$ into $WAP(\mA)= \mA\cdot \mA^*.$ Since it is clear that
$R^*_m( \varphi\cdot a)= R^*_m( \varphi) \cdot a,$
\cite[Theorem 3.1]{balapy} gives an element $a_0\in \mA$ such that \[m\cdot \varphi= R^*_m( \varphi)=
a_0\cdot \varphi\quad\text{for every}\quad \varphi\in \mA^*.\]
In other words, \[\langle m,\varphi\cdot a\rangle= \langle m\cdot\varphi,a\rangle=
\langle a_0\cdot\varphi,a\rangle= \langle  a_0, \varphi\cdot a\rangle\quad\text{for every}\quad
\varphi\cdot a\in WAP(\mA).\]
Now if $m$ is decomposed as $m=em+(1-e)m$ for some right identity $e$ in $\mA^{**},$
then \[\langle em,\varphi\rangle=\lim_\alpha\langle m,\varphi\cdot e_\alpha\rangle=\lim_\alpha
\langle a_0, \varphi\cdot e_\alpha\rangle=\langle ea_0, \varphi\rangle=\langle a_0,\varphi\rangle.\]
This shows the claim.

The assumption that $ \mA^*  \cdot m\subseteq WAP(\mA)$ yields the same claim with an analogous argument.
 %[$R^*_m(\varphi\cdot a)=R^*_m(\phi) \cdot a$],
\end{proof}

%Our next results show how close small-1-1 ideals are  to constitute a boundary between Arens regular and Arens irregular algebras.

\begin{proposition}\label{wsc}
Let $ \mA $ be a  Banach algebra in the  \wasabi  class.
Let $\mJ$ be a closed ideal in $\mA$.
If $\mJ$ is a  weakly Riesz ideal (equivalently, an Arens regular ideal by Theorem \ref{prop1}, \emph{infra}) , then it is a small-1-1 ideal.
\end{proposition}

        \begin{proof} Let $\epsilon>0, $
         $m,n\in \mJ^{**}$ and $\varphi\in \mA^*$. Let, as usual,
          $i^\ast \mA^\ast \to \mJ^\ast$ denote the   restriction mapping that is adjoint to the inclusion $i\colon \mJ\to \mA$. Since $\mJ$ is a weak Riesz ideal, we may pick $a_\varepsilon\in \mJ$ and $\psi_\varepsilon \in \mA^\ast$
           such that
\[ \norm{n\cdot i^\ast(\varphi)-a_\varepsilon i^\ast(\psi_\varepsilon)}_{\mJ^\ast}<\frac{\varepsilon}{\norm{m}}.\]
If now $b$ is an arbitrary element of $\mA$ with $\norm{b}\leq 1$,
% and $(m_\alpha)$ is a net in $\mJ$ with weak$^\ast$-limit $m$ (we can assume that  $\norm{m_\alpha}\leq \norm{m}$),
then, taking into account that $bm\in \mJ$ and that $ i^*(m\cdot\varphi)=m\cdot i^*(\varphi)$,
\begin{align*}
  |\<{mn\cdot \varphi-ma_\varepsilon\cdot  \psi_\varepsilon,b}|&=|\<{m\cdot(n\cdot \varphi-a_\varepsilon\cdot \psi_\varepsilon),b}|\\
  %&=\<{m,(en\phi-a_\varepsilon\psi_\varepsilon)\cdot b}\\
 % &=\lim_\alpha\<{m_\alpha, (en\phi-a_\varepsilon\psi_\varepsilon)\cdot b}\\
%  &=\lim_\alpha\<{(en\phi-a_\varepsilon\psi_\varepsilon),bm_\alpha}\\
%  &=\lim_\alpha\<{eni^\ast(\phi)-a_\varepsilon i^\ast(\psi_\varepsilon),bm_\alpha}\\
  &=|\<{n\cdot i^\ast(\varphi)-a_\varepsilon \cdot i^\ast(\psi_\varepsilon),bm}|\\
  &\leq \norm{n\cdot i^\ast(\varphi)-a_\varepsilon \cdot i^\ast(\psi_\varepsilon)}_{\mJ^*}\norm{m}_{\mJ^{**}}<\varepsilon.
\end{align*}
 It follows that $ \lim_{\varepsilon\to 0} \norm{mn\cdot \varphi-ma_\varepsilon\cdot \psi_\varepsilon}_{A^\ast}=0$. Since $ma_\epsilon\in\mJ,$ we find that $ma_\varepsilon\cdot \psi_\varepsilon\in WAP(\mA)$ for every $\varepsilon>0$, and we conclude that $mn\cdot \varphi\in WAP(\mA)$.
%
%           \begin{align*}
%             \left|\<{emn\cdot \varphi-ema_\varepsilon\cdot \psi_\varepsilon,b}\right|&=\left|\<{n\cdot \varphi- a_\varepsilon\psi_\varepsilon,bem}\right|
%             \\&=\left|\<{i^\ast(n\cdot \varphi)- i^\ast(     a_\varepsilon\psi_\varepsilon),bem}\right|\\&\leq  \norm{n\cdot i^\ast(\varphi)-a_\varepsilon\cdot i^\ast(\psi_\varepsilon)}\cdot \norm{bem}\leq \varepsilon\cdot \norm{em}.
%           \end{align*}
%           It follows that $ \lim_{\varepsilon\to 0} \norm{emn\cdot \varphi-ema_\varepsilon\cdot
%\psi_\varepsilon}_{A^\ast}=0$. Since $ema_\varepsilon\cdot \psi_\varepsilon\in WAP(\mA)$ for every $\varepsilon>0$ , we conclude that $emn\cdot \varphi\in WAP(\mA)$.
        We apply then Lemma \ref{lem5} to obtain that $mn =a_0 +r$ for some $a_0\in \mA$ and $r\in WAP(\mA)^\perp$.    This immediately implies that  $emn =a_0\in \mJ$, as required.
\end{proof}

%
%\begin{lemma}
%Let $ \mA $ be a  Banach algebra in the  \wasabi  class\footnote{\tt The old statement required a bai, wouldn't a brai be enough? should  we define  the \wasabi class requiring justa brai?} .
%Let $\mJ$ be a closed ideal in $\mA$.
%The ideal $\mJ$ is   weakly Riesz if and only if $e\mJ^\bd\mJ^\bd\cdot \mA^\ast\subseteq WAP(\mA)$.
%\end{lemma}

\section{Arens regularity}
We deal in this section with Arens regularity of closed ideals $\mJ$. The first theorem determines precisely when $\mJ$ is Arens regular. It shows in particular that weak Riesz ideals characterize Arens regular ideals in any Banach algebra, which is an ideal in its second dual $ \mA^{**}$ and has a bai. The equivalence of the first two statements of the theorem was
%theorem and the consequences were
proved in \cite{EFG} for closed ideals in the group algebra of a compact Abelian group.

\begin{lemma}\label{reg} Let $\mA$ be a two-sided ideal in $\mA^{**}$. Then for every $m,n,p\in\mA^{**}$ and $\varphi\in \mA^*$, we have
\begin{enumerate}
\item  $(p\cdot \varphi)\cdot m=p\cdot(  \varphi\cdot m).$
\item $(m\diamondsuit n)p=m\diamondsuit ( np).$
\end{enumerate}
%{\tt If $\mA$ is a right ideal ...}
\end{lemma}

\begin{proof} For the first statement,
let $a\in\mA.$ Then, using the facts that  $ma\in \mA$ and $ap\in \mA$,  we find
\begin{align*}
\langle (p\cdot \varphi)\cdot m, a\rangle&=\langle p \cdot \varphi, m\diamondsuit a\rangle
=\langle p \cdot \varphi, m a\rangle\\&=\langle (ma) p,  \varphi\rangle=
\langle m(a p),  \varphi\rangle=\langle m\diamondsuit (a p),  \varphi\rangle\\&=
\langle a p,  \varphi\cdot m\rangle=
\langle a, p\cdot(  \varphi\cdot m)\rangle,\end{align*}
and so $(p\cdot \varphi)\cdot m=p\cdot(  \varphi\cdot m),$ as required.

The second statement follows from the first as
 \begin{align*}\langle(m\diamondsuit n)p, \varphi\rangle&=
\langle m\diamondsuit n, p\cdot \varphi\rangle= \langle n, (p\cdot \varphi)\cdot m\rangle\\&=\langle n, p\cdot(  \varphi\cdot m)\rangle= \langle n p,   \varphi\cdot m\rangle=
\langle m\diamondsuit(n p),   \varphi\rangle.\end{align*}
\end{proof}

%\textcolor{red}{\tt I made economies in the following proof by exchanging the roles of (ii) and (iii)}
%
%{\tt Another way of arguing for next theorem to prove (i) implies (iii) directly}

%Suppose that $\mJ$ is not a weak Riesz ideal. Pick $m\in\mJ^{**}$ and $\varphi\in\mJ^*$ such that $m\cdot \varphi\notin\overline{i^*(WAP(A))}=\overline{\mA\cdot\mJ^*}.$ Take by Hahn-Banach, $n\in \mJ^{**}$ such that $\langle n,\mA\cdot\mJ^*\rangle=\{0\}$ and
%$\langle n, m\cdot\varphi\rangle=1.$
%Then $\langle n a, \varphi\rangle=0$ for every $a\in\mA$ and $\langle n m, \varphi\rangle=1$ implies that $\mJ$ is not Arens regular.
%

\begin{theorem} \label{prop1}
Let $ \mA $ be a Banach algebra, which is an ideal in its second dual $ \mA^{**}  $
and has a    bai. Let $ \mJ $ be a closed ideal  of $ \mA$.  Then the following statements are equivalent.
\begin{enumerate}
\item $\mJ  $ is Arens regular.
%$\overline{e\mJ^{**} \cdot \mJ^*}\;\cup\;\overline{\mJ^* \cdot \mJ^{**}\diamondsuit %e}\subseteq\overline{i^*(WAP (\mA))}.$ {\tt equality would have been nice!}
%are both contained in  $\overline{i^*(WAP (\mA))}$
% (the norm closure in $\mA^*$).
%\item $\mJ  $ is an Arens ideal in $\mA$.
\item $i^*(WAP(\mA))^\perp\subseteq \mathcal Z_1(\mJ^{**})\cap  \mathcal Z_2(\mJ^{**}).$
\item  $\mJ$ is a weak Riesz ideal in $\mA.$
\item $i^*(WAP(\mA))^\perp \mJ^\bd=\{0\}$ and $\mJ^\bd\diamondsuit i^*(WAP(\mA))^\perp =\{0\}.$
    \end{enumerate}
\end{theorem}

\begin{proof}
There is no need to prove that (i) implies (ii).

%Suppose  that $ \mJ $ is Arens regular and  let  $ r \in i^*(WAP(\mA))^\perp$, $m\in \mJ^{**}$ and $\varphi\in \mJ^*$.   Then,
%\begin{align*}
%&\langle r ,  m \cdot \varphi \rangle
%= \langle r m  , \varphi  \rangle
%= \langle r \diamondsuit m  , \varphi  \rangle
%=0,\\
%&
%\langle r , \varphi \cdot m  \rangle
%= \langle m \diamondsuit r  , \varphi \rangle
%= \langle m   r , \varphi  \rangle
%=0.
%\end{align*}
%Thus, $ \mJ^{**}\cdot\mJ^*$ and $\mJ^*\cdot\mJ^{**}$ are both contained in ${i^*(WAP(\mA))}^{\perp\perp}=\overline{i^*(WAP(\mA))}$.

%Now $\overline{e\mJ^{**} \cdot \mJ^*}$ must be equal to $\mJ^*$. For otherwise,
%there exists a nonzero element $m\in \mJ^{**}$ which annihilates $e\mJ^{**} \cdot \mJ^*$. Then,

We now prove that (ii) implies (iii).
So, let's assume that $i^*(WAP(\mA))^\perp\subseteq \mathcal Z_1(\mJ^{**})$. Then,
for arbitrary $r\in i^*(WAP(\mA))^\perp$, $m\in \mJ^{**}$ and $\varphi\in \mJ^*$,
we have \[\langle r, m\cdot\varphi\rangle=\langle rm,\varphi\rangle=\langle r\diamondsuit m,\varphi\rangle=0,\]
and so \[m\cdot\varphi\in i^*(WAP(\mA))^{\perp\perp}=\overline{i^*(WAP(\mA))}.\]
Similarly, $\varphi\cdot m\in \overline{i^*(WAP(\mA))}$ when
$ i^*(WAP(\mA))^\perp\subseteq  \mathcal Z_2(\mJ^{**})$.

To see that  (iii) implies (i), we
let $m_1$ and $m_2$ be arbitrarily chosen in $\mJ^{**}$.

Then, for every $\varphi\in \mJ^*,$ we have
 \begin{align*}\langle m_1 m_2,\varphi\rangle&=\langle m_1, m_2\cdot\varphi\rangle
=\langle m_1\diamondsuit e, m_2\cdot\varphi\rangle=\langle (m_1\diamondsuit e) m_2,\varphi\rangle\\&=\langle(m_1\diamondsuit( e m_2), \varphi\rangle=\langle( e m_2), \varphi\cdot m_1\rangle\\&=
\langle m_2, \varphi\cdot m_1\rangle=\langle m_1\diamondsuit m_2, \varphi\rangle,\end{align*}
where the second and sixth equalities are due to the assumption that \[{\mJ^{**} \cdot \mJ^*}\;\cup\;{\mJ^* \cdot \mJ^{**}}\subseteq\overline{i^*(WAP (\mA))},\] and the fourth equality is due to Lemma
\ref{reg}.
Therefore, $m_1m_2=m_1\diamondsuit m_2$ and Arens regularity of $\mJ$ follows.

%To see that (ii) implies (iii), assume that $\mJ^{**} \cdot \mJ^*\;\cup\;\mJ^* \cdot \mJ^{**}\subseteq \overline{i^*(WAP(\mA))}$. Let  $ r \in {i^*(WAP(\mA))}^\perp $ and  $m\in \mJ^{**}$ be arbitrary.
%Then for every  $\varphi\in \mJ^*$, we have
%\[\langle rm, \varphi\rangle =\langle r,m\cdot\varphi\rangle=0\]
%since $m\cdot\varphi\in \overline{i^*(WAP(\mA))} $ by assumption.
%Accordingly,
% \[rm=0=r\diamondsuit m.\]
%%since the elements in  $WAP(\mA)^\perp$ are right annihilators of $\mJ^{**}$.
% Therefore,  $i^*(WAP(\mA))^\perp\subseteq \mathcal Z_1(\mJ^{**}).$
%
%
%Similarly, $\varphi\cdot m\in\overline{ i^*(WAP(\mA))}$ for every $\varphi\in \mJ^*$ yields $m\diamondsuit r=0=mr,$ showing that $r\in \mathcal Z_2(\mJ^{**}).$
% Therefore,  \[i^*(WAP(\mA))^\perp\subseteq \mathcal Z_1(\mJ^{**})\cap \mathcal Z_2(\mJ^{**}),\]
% as required.

We have now checked that Statements (i) through (iii) are equivalent.

Since elements of $i^\ast(WAP(\mA))^\perp$ are right annihilators for the first Arens product and left annihilators for the second, we see that statements  (ii) and (iv) are clearly equivalent and the theorem is proved.
\end{proof}

A  Riesz ideal in $\mA$ is clearly Arens regular. This can be either deduced directly
from the decomposition (\ref{direct}), or from the theorem above, for  Riesz ideals are weakly Riesz.
  When $\mA$ is the group algebra of a compact Abelian group,  this result was  obtained originally by \"Ulger in \cite{U} 	and proved again by the authors in \cite{EFG}.
	
 In addition to the previous theorem and Proposition \ref{wsc},
our next results show how close small-1-1 ideals are  to constitute a boundary between Arens regular and Arens irregular algebras.

%
%Nevertheless,  our next theorem gets very close to a positive answer instead for a large class of Banach algebras wich includes the  group algebra $L^1(G)$
%for compact groups and the Fourier algebra $A(G)$ for   discrete amenable groups.

%\begin{align*}
    %S^{(l)}_e(J) &=\{r   \fR_e(T): r \in i^*(WAP(\mA))^\perp, T \in \wJ \};\\
   % S^{(r)}_e(J) &=\{  \fL_e(T)\diamondsuit r: r \in i^*(WAP(\mA))^\perp, T \in \wJ \}.
%\end{align*}

\begin{theorem}  \label{thm2}
Let $ \mA $ be a  Banach algebra with a   bai, and suppose that it is an ideal in its second dual $ \mA^{**}  $.
Let $ \mJ $ be a  small-1-1 ideal of $\mA$.  Then
%, for any mixed identity $e \in \mA^{**}$:
\[ \mJ^{**}\mJ^{**} \subseteq \mathcal{Z}_1(\mJ^{**})\quad\text{and}\quad
 \mJ^{**}\diamondsuit\mJ^{**} \subseteq \mathcal{Z}_2(\mJ^{**}).\]
 In particular, $\mJ$ is  sAir  if and only if it is reflexive.\end{theorem}

\begin{proof}
 We show first that $i^*(WAP(\mA))^\perp\mJ^{**}\subseteq \mathcal{Z}_1(\mJ^{**}) .$
Fix a mixed identity $e$ of $\mA^{**}$. Let
 $p=rm$  with $m\in\mJ^{**} $ and $r \in i^*(WAP(\mA))^\perp$.
Then for every $n\in \mJ^{**}$ and $s \in i^*(WAP(\mA))^\perp$, we find
\[(em)(en)= e(mn) \in \mJ\subseteq \mathcal {Z}_1(\mJ^{**}),\] and so
\[     p (en+s) = (rm) n = r(emn) =0.\]
On the other hand,  since $\varphi \cdot (rm )=0$  for all  $\varphi \in \mA^*$, we have
\[
    p \diamondsuit (en+s) = p \diamondsuit (en+s) =0.
\]
 Hence $p \in \mathcal {Z}_1(\mJ^{**})$ and the containment  $i^*(WAP(\mA))^\perp\mJ^{**}\subseteq \mathcal{Z}_1(\mJ^{**}) $ is proved.

 Since, by the decomposition of Theorem \ref{8},  \begin{equation}\mJ^{**}\mJ^{**} \subseteq \mJ+i^*(WAP(\mA))^\perp\mJ^{**}, \label{sair}\end{equation} the first inclusion in the theorem follows.

 The inclusion $\mJ^{**}\diamondsuit\mJ^{**} \subseteq \mathcal{Z}_2(\mJ^{**})$
is proved in the same manner.

Finally, since $i^*(WAP(\mA))^\perp\mJ^{**}\subseteq \mJ^\bd\mJ^\bd\subseteq\mathcal Z_1(\mJ^\bd)$
and   $WAP(\mA)^\perp \mA^\bd \cap \mA=\{0\}$ , we see that $\mJ$ can be  lsAir  only when
\begin{equation*}  i^*(WAP(\mA))^\perp\mJ^{**}=\{0\}.\end{equation*}
In the same way, $\mJ$ can only be rsAir   when
\begin{equation*}\label{wapj**} \mJ^{**}\diamondsuit i^*(WAP(\mA))^\perp=
\{0\}.\end{equation*}

Proposition \ref{prop1} then shows that $\mJ$ must be Arens regular and, hence, reflexive.
%
%But \eqref{wapj**} implies that $m m^\prime=emm^\prime\in \mJ$ and $m\diamondsuit m^\prime=m\diamondsuit m^\prime\diamondsuit e\in \mJ$ for every $m,m^\prime\in \mJ^\bd$. We would then have that
% $mm^\prime-m\diamondsuit m^\prime \in i^*(WAP(\mA))^\perp\cap \mJ$, which is only possible if  $mm^\prime-m\diamondsuit m^\prime =0$, for each, $m,\,m^\prime \in \mJ^\bd$. So,    $\mJ$ must be Arens regular.

\end{proof}

As a summary of these regularity-like properties, we see that
 Theorem \ref{prop1} tells us   that Arens regular ideals of \wasabi algebras are exactly the same as weak Riesz ideals and  Proposition \ref{wsc} that they are always  small-1-1. After Theorem \ref{prop1},   we have therefore that for a nonreflexive ideal
\[\text{Riesz}\Longrightarrow \text{weakly Riesz} \Longleftrightarrow \text{Arens regular} \Longrightarrow \text{small-1-1}\Longrightarrow \text{Not sAir}.\]
It could be worth to observe that weak sequential completeness is only needed for the penultimate  implication.

%{\tt I marked this remark with tt, it seemed to me more suitable for our own reflection
%\begin{remark}{\tt On one hand, small-1-1 ideals are  very close to be Arens regular since  $a\mJ^{**}=ea\mJ^{**}\subseteq \mJ$ and so $a\mJ^{**}\subseteq \mathcal Z_1(\mJ^{**})$ for every $a\in \mJ$.
%
%On the other hand,  this theorem cannot be applied directly to show $\mJ$ is Arens regular.
%For this to happen, one needs $i^*(WAP(\mA))^\perp\mJ^{**}=\{0\}$, which is the same as
%$\mJ^{**}\cdot\mJ\subseteq\overline{WAP(\A)}.$ In other words, we are back to square one:
%are small-1-1 ideals weak Riesz ideals?}
%\end{remark}
%}

%\end{document}

\section{Strong Arens irregularity}
In this section we deal with the strong Arens irregularity of closed ideals of Banach algebras  in the \wasabi class.
%, i.e., of Banach algebras which are weakly sequentially complete, have  a  bai and are an ideal in their  second dual.

Algebras in the \wasabi class are always strongly Arens irregular, \cite[Theorem 2.2]{balapy}.  In the particular case of  the group algebra $L^1(G)$ of a compact group $G$, this was  proved by Isik-Pym-\"Ulger in \cite{IPU87}.  Lau-Losert  \cite{LL} obtained the same result for   the Fourier algebra $A(\Gamma)$ of a discrete amenable group $\Gamma$.  Bounded approximate identities were essential in proving these results. We now see that on ideals of algebras in this class, which may not have bounded approximate identities,  a weaker assumption yields the same result.
%We deal with strong Arens irregularity of closed ideals $\mJ$ of \wasabi Banach algebras.
%We first  determine precisely when $\mJ$ is strongly Arens irregular.
A different approach, based on \cite{IPU87},
 was used in   \cite{EFG} to get this result for  closed ideals in the group algebra of a compact Abelian group.

\begin{lemma} \label{lemm6} \label{lem6}Let $ \mA $ be a Banach algebra  and let $ \mJ $ be a  closed  ideal of $\mA$.
 \begin{enumerate}
\item If the linear span of $\mJ^{**} \cdot \mJ^*  $ is dense in $\mJ^*, $ then $\mathcal Z_1(\mJ^{**})\subseteq \mathcal Z_1(\mA^{**}).$
\item  If  the linear span of  $  \mJ^* \cdot \mJ^{**} $  is dense in $\mJ^*  $,
then $\mathcal Z_2(\mJ^{**})\subseteq \mathcal Z_2(\mA^{**}).$
\item  If  the linear span of    $\mJ^{**} \cdot \mJ^* \cdot \mJ^{**} $    is dense in $\mJ^*  $,
then $\mathcal Z_i(\mJ^{**})\subseteq \mathcal Z_i(\mA^{**}),$ $i=1,2.$
\end{enumerate}
\end{lemma}

\begin{proof}
We shall prove the first statement, the  second statement is proved in the same manner.
Let $m\in \mathcal Z_1 (\mJ^{**})$, $p\in \mA^{**}$ and
$\varphi\in \mA^{*}$. Assume first that  $i^*(\varphi)=n\cdot i^*(\psi)$ for some $\psi\in \mA^*$ and $n\in \mJ^{**}$. Then,  we find
\begin{align*}\langle mp, \varphi\rangle &=\langle mp, i^*(\varphi)\rangle
=\langle mp,  n\cdot i^*(\psi)\rangle=
\langle (mp)n, i^*(\psi)\rangle=\langle m(pn), i^*(\psi)\rangle
\\&=\langle m\diamondsuit (pn), i^*(\psi)\rangle
=\langle (m\diamondsuit p)n, i^*(\psi)\rangle
=\langle  m\diamondsuit p, n\cdot i^*(\psi)\rangle
\\&=\langle m\diamondsuit p, i^*(\varphi)\rangle
=\langle m\diamondsuit p, \varphi\rangle,\end{align*}
 using the fact that $\mJ^{**}$ is an ideal in $\mA^{**}$ for the first and last equalities,  and $m\in \mathcal Z_1 (\mJ^{**})$ for the fifth and Lemma \ref{reg} for the sixth.

The condition on the linear span  of $\mJ^{**} \cdot \mJ^*  $ being dense in $\mJ^* $
yields then \begin{equation*}mp=m\diamondsuit p\quad\text{for every}\quad p\in \mA^{**},\end{equation*}
that is, $m\in Z_1(\mA^{**}).$ Thus  $\mathcal Z_1 (\mJ^{**})\subseteq \mathcal Z_1 (\mA^{**}).$

To prove Statement (iii), recall that $(\mJ^{**} \cdot \mJ^*) \cdot \mJ^{**}=
\mJ^{**} \cdot( \mJ^* \cdot \mJ^{**})$ by Lemma \ref{reg}(i).
So, if the linear span  of   $\mJ^{**} \cdot \mJ^* \cdot \mJ^{**}  $    is dense in $\mJ^*  $,
then each of the linear spans of $\mJ^{**} \cdot \mJ^*  $ and $  \mJ^* \cdot \mJ^{**} $ is dense in $\mJ^*  $.
Statements (i) and (ii) imply therefore  that  $\mathcal Z(\mJ^{**})\subseteq Z(\mA^{**}).$
\end{proof}

Baker, Lau and Pym proved in \cite[Theorem 2..2(ii)]{balapy} that $\mathcal Z_1(\mA^{**})=\mA$ for \wasabi algebras.
Their proofs can also be applied to show that $\mathcal Z_2(\mA^{**})=\mA$. So  \wasabi algebras are sAir by
\cite[Theorem 2..2(ii)]{balapy}. Our next theorem characterizes the closed ideals in \wasabi algebras which are lsAir, rsAir or sAir.

\begin{theorem}  \label{lem6}Let $ \mA $ be a \wasabi algebra  and let $ \mJ $ be a  closed  ideal of $\mA$.
 \begin{enumerate}
\item   If $\mJ$  is left faithful, then $\mJ$ is lsAir if and only if the linear span of $\mJ^{**} \cdot \mJ^*  $ is dense in $\mJ^*.$
\item  If $\mJ$  is right faithful, then $\mJ  $ is rsAir if and only if the linear span of  $\mJ^* \cdot \mJ^{**} $  is dense in $\mJ^*  $.
\item If $\mJ$  is either left faithful or right faithful, then $\mJ$ is sAir if and only if the linear span of  $\mJ^{**} \cdot \mJ^* \cdot \mJ^{**} $ is dense in $\mJ^*.$
\end{enumerate}
\end{theorem}

\begin{proof} We prove Statement (i), the second statement follows in the same way.
Suppose that the linear span of $\mJ^{**} \cdot \mJ^*  $ is dense in $\mJ^*.$
Then by lemma \ref{lemm6}, $\mathcal Z_1(\mJ^{**})\subseteq \mathcal Z_1(\mA^{**}).$
Since $\mA$ is lsAir by \cite{balapy}, we conclude that $\mathcal Z_1 (\mJ^{**})= \mJ$.

For the converse, suppose that the linear span of $  \mJ^{**} \cdot \mJ^{*} $  is not dense in $\mJ^*  $
and pick a  nonzero $m\in \mJ^{**}$  which is zero on $  \mJ^{**} \cdot \mJ^{*} $.
Then clearly $mn=0$ for every $n\in \mJ^{**}$.
On the other hand, \[\langle\varphi\cdot m, a\rangle=\langle m,a\cdot\varphi\rangle=0\quad
\text{}\quad \varphi\in \mJ^*\quad \text{and}\quad a\in \mJ.\]
It follows that $\varphi\cdot m=0$, and so
$\langle m\diamondsuit n,\varphi\rangle=\langle n, \varphi\cdot m\rangle=0.$
Therefore \[mn=m\diamondsuit n=0\quad\text{ for every}\quad\in \mJ^{**},\] which means that  $m\in \mathcal Z_1(\mJ^{**})$.
Note that $m$ cannot be in $\mJ$  since $m\mJ=\{0\}$ and $\mJ$ is left faithful. Thus $\mJ$ is not lsAir, as required.

As for Statement (iii), we apply again Lemma \ref{lemm6} and \cite{balapy}. We find $\mJ$ is sAir when the linear span of $\mJ^{**} \cdot \mJ^* \cdot \mJ^{**}$ is dense in $\mJ^*$.
Conversely, let $\mJ$ be left faithful,  suppose that it  is sAir, and let $ r \in (\mJ^{**} \cdot \mJ^* \cdot \mJ^{**})^\perp $ and $m  \in \mJ^{**} $. Then for each  $ \varphi \in \mJ^* $ and $ n \in  \mJ^{**}$, we have
\begin{align*}
    \langle m \diamondsuit r , n \cdot \varphi\rangle
    =\langle  r , n \cdot  \varphi\cdot m\rangle
    =0.
\end{align*}
Since $\mJ $  is lsAir, the linear span of $\mJ^{**} \cdot \mJ^*  $ is dense in $\mJ^*$
 by Statement (i). So $m \diamondsuit r=0$. On the other hand,
$\langle mr, n\cdot \varphi\rangle=\langle m,r\cdot( n\cdot\varphi)\rangle$, where \[\langle r\cdot(n\cdot\varphi), a\rangle=
\langle  r, n\cdot\varphi\cdot a\rangle=0\quad\text{for all}\quad a\in\mJ.\]
Thus $  m r =0 $ on $\mJ^{**}\cdot \mJ^*,$ and so $mr=0$ by Statement (i) since $\mJ $  is lsAir. Therefore,
$ m  r =m \diamondsuit r=0, $ showing first that $ r \in  \mathcal Z_{2} (\mJ^{**}).$ Since $\mJ$ is rsAir, $r$ must be therefore in $\mJ.$
Since in particular $\mJ r=\{0\}$ and $\mJ$ is left faithful, we deduce that $r=0.$

In the same manner, using the fact that $\mJ$ is rsAir we may verify that $r\diamondsuit m= rm=0$., and so $r\in \mJ$ since $\mJ$
is lsAir. In particular, $r\mJ=\{0\}$ yields $r=0$ when $\mJ$ is right faithful.

In either cases, we conclude that $\mJ^{**} \cdot \mJ^* \cdot \mJ^{**}  $  must be dense in $\mJ^*  $.
\end{proof}

When $G$ is a locally compact Abelian group, a closed ideal of $L^1(G)$ has a bai if and only if it is  the kernel
of a  closed set in the Boolean ring $\Omega_{{\widehat G}_d}$ generated by the left cosets of subgroups of the group ${\widehat G}$ with the discrete topology, known as the coset ring
of $\widehat G$.
  This characterization was proved by Reiter
   in \cite{reit72} and by Liu, van Rooij and Wang  in \cite{liurooijwang73}. See also \cite[Section 5.6]{kanibook}, and \cite{K2000} for more developments and references.

For general compact groups,  Rider \cite{rider70} defined the analogue of $\Omega_{{\widehat G}}$ and proved,  [Theorem 1, loc. cit.],   that one can associate to  every  $E\in \Omega_{{\widehat G}}$   a central idempotent measure $\mu\in M_E(G)$. It follows from \cite[Theorem 2]{liurooijwang73} that $L_E^1(G)$ has a bounded approximate identity whenever $E\in \Omega_{{\widehat G}}$.

For a class of non-Abelian compact groups which  includes all  connected groups,  the converse is also true, \cite{rider73}.

The  dual of Reiter-Liu-van Rooij-Wang's theorem, regarding, $A(G)$, was proved in  \cite{forrkanilauspro03}:
a closed ideal of the Fourier algebra $A(G)$ of  a locally compact amenable group $G$ has a bai if and only if it consists  of the functions that vanish on a   closed set in the Boolean ring $\Omega_{G_d}$ generated by the left cosets of subgroups of
$G_d$.

We summarize these results in our context.

\begin{corollary} ~\label{cor:SAIBAI}
\begin{enumerate}
\item If $G$ is a compact group and $E\in  \Omega_{{\widehat G}}$, then $L_E^1(G)$ is sAir.
\item If $\Gamma$ is  a discrete amenable group and  $E\in \Omega_{\Gamma}$, then $A_E(\Gamma)$ is sAir.
\end{enumerate}
\end{corollary}

\begin{remark} Strongly Arens irregular ideals without bai are easy to construct.
For example, let $\mJ$ be as in Corollary \ref{cor:SAIBAI}, i.e. a closed ideal with a bai in a \wasabi algebra $\mA$ and let $\mB$ be a reflexive Banach algebra without identity  and consider
the closed ideal $\mJ\oplus_{1}\mB$ in $\mA\oplus_1\mB$
with the natural  product
\begin{align*}
    (a_1,b_1)(a_2,b_2)=(a_1a_2, b_1b_2).
\end{align*}
Then $\mJ\oplus_{1}\mB$ has no bai and the topological center of  $(\mJ\oplus_{1}\mB)^{**}=\mJ^{**}\oplus_{1}\mB$   equals  $\mJ\oplus_{1}\mB$, i.e., the ideal is sAir.  In Corollary \ref{cor:sum} we will develop this idea to find such ideals \emph{within} the original algebra $\mA$.
\end{remark}

\section{Extreme non-Arens regularity}
As seen in Proposition \ref{wsc}, ideals of algebras in the \wasabi class which are not small-1-1
are not Arens regular. We prove in this section that, at least when the bai is sequential, they are in fact enAr.
For this, we check that these ideals satisfy the  conditions required in  \cite [Corollary 3.10]{filagali21}
for a Banach algebra to be enAr.

The reader is directed for \cite{filagali21} or \cite{filagali22}
for definitions required in the proofs.
\begin{definition}\label{lone}
{\sc $\ell^1$-base}: Let $E$ be a normed space.   A bounded sequence  $B=\{a_n\colon n\in \N\}$ is  an \emph{$\ell^1$-basis} in $E$, with constant $K>0$,
when the inequality
\[\sum_{n=1}^p |z_k|\le K\left\|\sum_{k=1}^p z_ka_{n_k}\right\|\label{lonecondition}\]
holds for all $p\in\N$ and for every possible choice of scalars $z_1,\ldots ,z_p$ and  elements $a_{n_1},\ldots ,a_{n_p}$ in $B.$
\end{definition}
For the next proof, we need to recall the well-known theorem of Rosenthal \cite[Theorem 1]{rosenthal}, which asserts that a bounded sequence in a Banach space, always has a subsequence that is either weakly Cauchy or is an $\ell^1$-basis.

\begin{theorem}
Let $ \mA $ be a non-unital   Banach algebra in the \wasabi class whose   bai $(e_n)_n$ is sequential,
and let $ \mJ$ be an ideal in $\mA$ that is not small-1-1. Then
%\begin{enumerate}
%\item
 there is a linear bounded map from the quotient space $\mJ^*/W AP(\mJ )$ onto $\ell^\infty.$
In particular, $\mJ$ is non-Arens regular.
If $\mJ$ is in addition separable, then $\mJ$ is enAr.
%\end{enumerate}
\end{theorem}

\begin{proof}
Let $e$ be a mixed identity associated with $(e_n)_n$
, and pick  $ p, q \in \mJ^{**}$ such that $epq\notin \mJ$.
Since $\mJ$ is an ideal in $\mA$, we see that the sequence $ (e_n  p q)_n $ is in $\mJ$. This   sequence   cannot  have any weakly Cauchy subsequence.
For,  otherwise,  there would exist a subsequence $ (e_{n_k}    pq)_k $ with weak limit $ a $ and, due to weak sequential completeness, $a\in  \mJ $, i.e.,
 $\lim_n e_{n_k}  p q= a $ in the weak topology $ \sigma( \mA , \mA^*) $ (the limit $a$ is in $\mJ$ since $\mJ$ is convex and so weakly closed).
So if $f$ is any mixed identity associated with $(e_{n_k})_k$, this implies that $f p q=a\in \mJ.$
Therefore, \[ e  p q = (ef)pq = e (fpq) =ea=a\in\mJ,\]
  whence a contradiction.
 Hence, by Rosenthal theorem we can assume, upon taking a suitable subsequence, that $ (e_n  p  q)_n $ is an $ \ell^1 $-base. Let then
\begin{align*}
A&= \{ a_n = e_{n}  p~ :  n \in \mathbb{N}\}\\
B&= \{ b_n =e_{n}   q ~:  n\in \mathbb{N}\}.
\end{align*}
The sets $A$ and $B$ are in $\mJ$ since $\mJ$ is an ideal in $\mA.$
Consider as well  the upper and lower triangles defined by $A$ and $B$, given by
 \begin{align*}T^{u}_{AB}&=\left\{a_n  b_m\colon \; n< m\right\}\; \mbox{ and } \\
T^{l}_{AB}&=\left\{a_n  b_m\colon\; m< n\right\}.\end{align*}
Define
\begin{align*}
X_1&= \{ x_{nm} = e_{n}  p  q~ : n<m\}\\
X_2&= \{ x_{nm} = e_{m}  p  q~ : m<n\}.
\end{align*}
Since $\mJ$ is an ideal in $\mA$, we see that  $e_{n}p\in \mJ$ for each $n\in\N$, and so we have
\begin{align*}
\lim_m \Vert x_{nm}- a_nb_m\Vert
&=\lim_m \Vert e_{n} p q- (e_{n}  p )  (e_{m}  q )\Vert\\&
\le \lim_m \Vert e_{n} p- (e_{n}  p )e_{m}\Vert\Vert q\Vert=0
\end{align*}
for each $ n \in \mathbb{N} $. Thus, \[\lim_m \Vert x_{nm}- a_nb_m\Vert=0\;\text{for each}\;  n \in \mathbb{N}. \] The same argument gives
\begin{align*}
\lim_n \Vert x_{nm}- a_nb_m\Vert=0\;\text{for each}\;  m \in \mathbb{N}.
%&=\lim_n \Vert e_{2m+1} pq- (e_{2n} p)  (e_{2m+1}  q )\Vert\\
%&\le &\lim_n \Vert e_{2m+1} p- (e_{2n} p)  (e_{2m+1}   )\Vert\Vert q\Vert
%\\=\lim_n \Vert e_{2m+1}   m   p- e_{2m+1}   m   p \Vert=0.
\end{align*}
 Hence the sets $ X_1,X_2 $,
 approximate, respectively, segments in $ T^u_{AB} $ and $ T^l_{AB} $ in the sense of \cite[Definition 3.3]{filagali21}. Since $ (e_n pq)_n $ is an $ \ell^1 $-base as well,  $ X_1 $ is horizontally injective according to \cite[Definition 3.2]{filagali21}).
Similarly, $ X_2 $ is vertically injective. Now, \cite [Corollary 3.10]{filagali21}   completes the proof.
\end{proof}

%\corollary
%Let $ G $ be a second countable compact group and $ E $ be a nonempty subset of $ \Sigma $. If $ %E $ is not a small-1-1 set, then $L^1_E(G) $ is ENAR.

\section{Concrete examples}
 In this section we deal with the most prominent  \wasabi algebras: group algebras  of compact groups and Fourier algebras of discrete amenable groups.    We will explore the existence, within these algebras,  of infinite dimensional ideals that are reflexive, of Arens regular ideals that are not  reflexive and of   strongly Arens irregular ideals  with and without bounded approximate identities.

All through this section $G$ will denote a compact group and $\Gamma $ a discrete group.   When $G$ (or $\Gamma$) are assumed to be commutative, this assumption will convey  that  $\Gamma=\widehat{G}$.

As already mentioned in Section 3, ideals in $L^1(G)$ can  be described in terms   of subsets $E$ of their dual object $\widehat{G}$, the set of all equivalence classes of irreducible  unitary representations of $G$.
% For a compact group $G$,  ideals in $L^1(G)$ will be described in terms of  their dual object $\widehat{G}$, the set of all equivalence classes of irreducible  unitary representations of $G$. As alreIf $\mu \in M(G)$ is a bounded measure,
% one can define the Fourier-Stieltjes transform $\widehat{\mu}$  which sends every $\pi \in \widehat{G}$ to an operator $\widehat{\mu}(\pi)$ on the Hilbert space $\h_\pi$ where $\pi$ is defined.
Every closed ideal $\mJ$ of $L^1(G)$ is of the form
 \[L_E^1(G)=\left\{f\in L^1(G)\colon \widehat{\mu}(\pi)=0, \mbox{ if } \pi\notin E\right\},\]
 for some $E\subseteq \widehat{G}$, see \cite[Theorem 38.7]{hewiross2}.

In the case of a discrete amenable group  $\Gamma$, an analogous description of ideals can be made.
Every subset of $\Gamma$ is then a set of synthesis by Theorem 6.1.1 of \cite{kanilaubook}. It follows then from \cite[Corollary 5.1.4]{kanibook} that
 every  ideal  $\mJ$  of $A(\Gamma)$ is of the form
\[ A_E(\Gamma)=\left\{u\in A(\Gamma)\colon u(\chi)=0, \mbox{ if } \chi\notin E\right\},\]
where $E=\Gamma\setminus \{s\in \Gamma\colon u(s)=0\mbox{ for every } u \in \mJ\}$, the complement of the zero-set of $\mJ$.

We now lay down the decomposition described in  Theorem \ref{8}  in the particular case of ideals in $L^1(G)$ and $A(\Gamma)$. This  decomposition follows directly from Theorem \ref{8} as soon as one takes into account that both algebras contain contractive approximate identities (see, e.g., \cite[Theorem 28.53]{hewiross2} and \cite[Theorem 2.7.2]{kanilaubook})  and that $\overline{L_E^1(G)}^w=M_E(G)$ and $\overline{A_E(\Gamma)}^w=B_E(\Gamma)$.
 \begin{corollary}\label{decs:concr}
  Let $G$ be a compact group and $\Gamma$ a discrete amenable group.
  \begin{itemize}
    \item For any  $E\subseteq \widehat{G}$,
    \[L_E^1(G)^\bd =\fR_e(M_E(G)) \oplus i^\ast(C(G))^\perp \cong M_E(G)\oplus i^\ast(C(G))^\perp\]
    \item For any  $E\subseteq \Gamma$,
    \[ A_E(\Gamma)^\bd=\fR_e(B_E(\Gamma))\oplus i^\ast (WAP(A(\Gamma))^\perp\cong B_E(\Gamma)\oplus i^\ast (WAP(A(\Gamma))^\perp.\]
      \end{itemize}
\end{corollary}

\subsection{Infinite dimensional reflexive ideals}\label{sec:7}
 To describe   reflexive ideals in $A(\Gamma)$, we will make use of    the non-commutative $L^p$-spaces,  $L^p(\Gamma)$, $p\geq 1$.
 If $\tau$ is the canonical trace in $VN(\Gamma)$, given by $\tau(T)=\<{T \delta_e,\delta_e}$ ($\delta_e\in \ell^2(\Gamma)$), the norm of $T\in L^p(\Gamma)$  is defined by
\[ \norm{T}_p^p=\tau(|T|^p)\]
  and  $L^p(\Gamma)$ is made of  operators in $VN(\Gamma)$ for which this norm is finite. If $\gamma\in \Gamma$,
the translation operator by $\gamma$, i.e., the convolution operator defined by $\delta_\gamma$, is always in $L^p(\Gamma)$.

We refer to \cite{pisixu03} for a  summary of properties of noncommutative $L^p$-spaces. In particular that  $L^1(\Gamma)$ is isometrically isomorphic to the predual of $VN(\Gamma)$, i.e., to $A(\Gamma)$.

We start our work with a definition.
\begin{definition}
  Let $G$ and $\Gamma$ be, as fixed above, a compact group and a discrete group, respectively.
  \begin{itemize}
    \item We say that $E\subseteq \widehat{G}$ is a $\Lambda(p)$-set, $p>0$, if     there are $0<q<p$ and $C>0$ such that
        \[\norm{P}_p\leq C\, \norm{P}_q,\]
        for every  linear combination $P$ of matrix coefficients of representations belonging to  $E$ i.e., for  every trigonometric polynomial $P\colon G\to \C$ supported on $E$.
        \item We say that $E\subseteq \Gamma$ is a $\Lambda(p)$-set, $p>2$, if     there are $1<q<p$ and $C>0$ such that
        \[\norm{f}_{\lpg}\leq C\, \norm{f}_{\lqg},\]
        for every linear combination $f$ of translation operators by elements of $E$.
  \end{itemize}
\end{definition}

The notion of $\Lambda(p)$-set was introduced by Rudin \cite{rudin60} for Abelian groups and by  Picardello \cite{pica73} for discrete noncommutative groups. Sections 0 and 1 in \cite{harc99} provide a good introduction to $\Lambda(p)$-sets. Note that a set that is $\Lambda(p)$ for some $p$ is automatically $\Lambda(q)$ for every $0<q<p$. It is a consequence of H\"older's inequality that   if $E$ is a $\Lambda(p)$-set, $p>1$, then  $L^p_E(G)=L^1_E(G)$ and $L^p_E(\Gamma)=L^1_E(\Gamma)$. The proof of this  can be found in \cite[Theorem 37.7]{hewiross2}. The same proof works in $L^p(\Gamma)$ replacing H\"older's inequality in $L^p(G)$ for H\"older's inequality in $L^p(\Gamma)$ (see the introduction of \cite{pisixu03}). In particular both $L^1_E(G)$ and $L^1_E(\Gamma)$ are reflexive if $E$ is a $\Lambda(p)$-set.

%\begin{lemma}\label{lambda1=ref,l1}
%If $E\subset \widehat{G}$, $G$ a compact group, then  $L_E^1(G)$ is reflexive if and only if $E$ is a $\Lambda(1)$-set.
%\end{lemma}
%\begin{proof}
%  Lemma 2 of \cite{bachebe74} proves that   $L_E^1(G)$ is reflexive when $E$ is a $\Lambda(1)$-set. The Corollary of \cite{hare88} shows that $E$ is a $\Lambda(1)$-set whenever $L^1_E(G)$ is reflexive.
%\end{proof}%
%Both \cite{bachebe74} and \cite{hare88} prove that for a given $0<p$ every $\Lambda(p)$-set is also a $\Lambda(p+\varepsilon)$-set for some $\varepsilon>0$, hence the reflexivity even of $L_E^1(G)$.

We summarize here our last  remark.

\begin{theorem}\label{lambda1=ref,ag}
Let  $E\subset \Gamma$  ($E\subset  \widehat{G}$). If  $E$
is a $\Lambda(p)$-set, $p>1$, then $A_E(\Gamma)$ (respectively $L_E^1(G)$) is reflexive.
  \end{theorem}

  In the case of compact  groups, this can be pushed to a characterization. Although the paper \cite{hare88} is on Abelian groups, the proof actually works for all compact groups.%
%
%We now can link these sets with reflexivity of $L_E^1(G)$ and $A_E(G)$. Hare  proved  in \cite{hare88} that  a $\Lambda(p)$-set in a discrete Abelian group, $p>0$ is  always a $\Lambda(p+\varepsilon)$-set for some $\varepsilon>0$. Bachelis and Ebestein \cite{bacheebe74} had proved the same earlier  for $p>1$. The proof still works  when $G$ is noncommutative and $E\subseteq \widehat{G}$. This readily implies that $L^1_E(G)$ is reflexive if $E$ is a $\Lambda(1)$-set. The converse is proved in Corollary of \cite{hare88}. The proof again works for noncommutative groups. This latter property is the one useful to us, both together  yield:

\begin{theorem}[Hare, \cite{hare88}]\label{lambda1=ref,l1}
If $E\subset \widehat{G}$, $G$ a compact group, then  $L_E^1(G)$ is reflexive if and only if $E$ is a $\Lambda(1)$-set.
\end{theorem}
% In that case the ideals of $L^1(G)$ that are reflexive have been characterized (in general, such a characterization is not available in the literature but examples abound).
% By a Theorem of Bachelis and Ebenstein \cite{bachebe74} (but see the clarifying
% improvement by Hare \cite{hare88}) every $\Lambda(p)$-set, $0<p<2$,  is also a $\Lambda(p+\varepsilon)$-set for some $\varepsilon>0$. Hence, if $E$ is a $\Lambda(1)$-set, then $L_E^1(G)=L_E^{1+\varepsil  on}(G)$ set for some $\varepsilon>0$ and $L_E^1(G)$ is reflexive.
%The converse being also true, \cite{hare88}, one has that $L_E^1(G)$ is reflexive if and only if $E\subseteq \widehat{G}$ is a $\Lambda(1)$-set.

While not  every algebra in \wasabi contains a reflexive  infinite dimensional ideal, see Theorem \ref{tall} below, many of them indeed do. A theorem of Picardello  to the effect that every discrete group contains an infinite $\Lambda(4)$-set can be used, along with  Theorem \ref{lambda1=ref,ag},  to see that Fourier algebras of discrete groups are among them.

\begin{theorem}[Theorem 1 of \cite{pica73}, see also \cite{boze73}]
Every discrete group $\Gamma$ contains an infinite subset $E$ such that $A_E(\Gamma)$ is reflexive.
\end{theorem}
In the case of $L^1(G)$ the situation is more complicated. Denote by $\widehat{G}_n$ the set of (equivalences classes of)  irreducible unitary representations of dimension $n$. We say that $G$ is \emph{tall} if $\widehat{G}_n$ is finite for every $n$.
\begin{theorem}\label{tall}
  Let $G$ be a compact group. If $G$ is not tall then $\widehat{G}$ contains an infinite subset $E$ such that $L_E^1(G)$ is reflexive.
  %A compact Lie group contains such an infinite subset if and only $G$ is tall. In particular, if $G$ is a simply connected simple Lie group, a reflexive ideal of $L^1(G)$ is necessarily finite dimensional.
  If $G$ is a compact semisimple Lie group, then a closed ideal of $L^1(G)$ is reflexive if and only if it has finite dimension.
\end{theorem}
\begin{proof}
  The first statement is proved in \cite[Corollary 2.5]{hutc77}.
   %Ceccini \cite[Theorem 3]{cecc72}

 Giulini and Travaglini prove in \cite[Corollary]{giultrav80} that compact semisimple Lie groups do not contain infinite  $\Lambda(q)$-sets for any $q>0$. We then  apply  Theorem  \ref{lambda1=ref,l1}.
%Rider \cite{rider75BUMI} proves that $SU(n)$ contains no infinite $\Lambda(p)$-set, $p>1$.    The second statement follows from this, Lemma \ref{1+ep} and
\end{proof}

\subsection{Arens-regular ideals that are not reflexive}
 Proposition \ref{prop1} shows that Riesz ideals of  \wasabi  algebras are Arens regular. An ideal $L_E^1(G)$ in $L^1(G)$ with $E\subseteq\widehat{G}$, is a Riesz ideal precisely when  $M_E(G)=L_E^1(G)$,  i.e., when $E$ is a \emph{Riesz} set.  The same can be said about ideals $A_E(\Gamma)$ in $A(\Gamma)$, cf. Section  \ref{special}

   %An easy way to obtain regular  When  a subset $E$ of the dual $\widehat{G}$ of a compact Abelian group $G$ is 'thin' enough, then $E$ is %a  Riesz set. But  thin sets are often  $\Lambda(1)$ so that  $L_E^1(G)$  is regular but also  reflexive. This notwithstanding,
%
  The eponymous theorem of the Riesz brothers at the origin of the term Riesz set,  shows that $L^1_\N(\T)$ is an  Arens regular ideal of $L^1(\T)$ and it is certainly  not reflexive. The literature contains several constructions of Riesz sets that are inspired by this classical result but  which do not part much from its original  spirit, see e.g., \cite{bochner,Mac}  and the survey paper \cite{koshi01}.
   Blei, see, e.g., \cite{blei75dio} constructs Riesz sets of a  different sort, they are, actually, sets with even stronger interpolation properties  that are not  $\Lambda(p)$ for any $p$. His examples, however, cannot be constructed in every Abelian group. We present here a rather general construction, partly based on Blei's,  that yields examples of Riesz sets  $E\subseteq \Gamma$, and hence  of  Arens regular ideals in $A_E(\Gamma)$, $\Gamma$ not necessarily commutative, so that $A_E(\Gamma)$ is not reflexive. When put together with known results, this will be used to show  that every Abelian discrete group contains a Riesz set (even a Rosenthal set) that is not $\Lambda(1)$.

 We begin with a few preliminaries, starting with the following Lemma
 that basically follows from \cite[Lemma 1.1]{gode88}. Recall (see \cite{kanilauschl03} or \cite{miao99}) that $B(\Gamma)=A(\Gamma) \oplus B_s(\Gamma)$ where $B_s(\Gamma)$ is a closed translation invariant subspace of $B(\Gamma)$. If $\phi=\phi_a+\phi_s\in B(\Gamma)$, then
     \[\norm{\phi}=\norm{\phi_a}+\norm{\phi_s}.\]
     This is the Lebesgue decomposition in $B(\Gamma)$.

%   \begin{lemma}\label{point}
%Then  either  $E$ is Riesz or there is $x\in E$ such that, putting $E_x=E\setminus\{x\}$,
%    there is $\phi \in B_{E_x}(\Gamma)$ with $\phi_s (x)\neq 0$.
%    %\notin $B_{E_x}(\Gamma)$.
%        \end{lemma}
%
%\begin{proof}											
%    If $E$ is not Riesz,
% there is  $\phi \in B_E(\Gamma)$ with  $\phi_s\neq 0$. We  pick $x\in G$ such that $\phi_s(x)\neq 0$ and define $\psi_x:=\phi-\phi(x)\delta_x\in B_{E_x}(\Gamma)$. If  $(\psi_x)_s(x)= 0$,  the equality $(\psi_x)_s=\phi_s$ yields the contradictory equality $0\neq \phi_s(x)=(\psi_x)_s(x)=0$.
%        \end{proof}
%				
%		{\tt The lemma and proof re-stated}	

	\begin{lemma}\label{point}
Let $\Gamma $ be a discrete group and $E\subseteq \Gamma$. Then either $E$ is a Riesz set or there exist $x\in E$
and
$\phi\in B_E(\Gamma)$ such that $\phi(x)=0$ and $\phi_s(x)\ne 0.$
%Then  either  $E$ is Riesz or there is $x\in E$ such that, putting $E_x=E\setminus\{x\}$,
%    there is $\phi \in B_{E_x}(\Gamma)$ with $\phi_s (x)\neq 0$.
    %\notin $B_{E_x}(\Gamma)$.
        \end{lemma}	
				
                \begin{proof}  If $E$ is not Riesz,
 there is  $\psi \in B_E(\Gamma)$ with  $\psi_s\neq 0$. We  pick $x\in E$ such that $\psi_s(x)\neq 0$ and define
$\phi:=\psi-\psi(x)\delta_x\in B_{E}(\Gamma)$. Clearly, $\phi(x)=0.$
If  $\phi_s(x)= 0$,  the equality $\phi_s=\psi_s$
yields the contradictory equality $0\neq \psi_s(x)=\phi_s(x)=0$.
\end{proof}

%\begin{lemma}[Lemma 3.1 of \cite{lauwong13}]\label{normu}
%Let $G$ be a locally compact group. Let $x\in G$, $U$ a neighbourhood of $x$ and $\varepsilon>0$. There is then another neighbourhood $V$ of $x$ with $V\subseteq U$ and $u\in A(G)$ such that
%\[ \restr{u}{V}=1, \quad  \restr{u}{G\setminus U}=0  \quad \mbox{ and } \quad \norm{u}\leq 1+\varepsilon.\]
%\end{lemma}
        The following definition marks the target of our construction.
				
        \begin{definition}\label{def:ce}
   Let $\Gamma$ be  a discrete group and $K$ be  a locally  compact group. We will say that a countable set $E\subset \Gamma$ is \emph{ convergently   embedded} in $K$ if there is a group isomorphism $\alpha \colon\<{E} \to K$ with dense range  so that $ \alpha(E)$ is relatively compact and has exactly only one limit point  $p\in K$, $p\notin \alpha(E)$.
   %\footnote{An alternative definition would be \emph{We will say that a countable set $E\subset \Gamma$ is \emph{ convergently   embedded} if $E$ is a convergent sequence in a locally precompact topology on $G$.} Which one is better in this context?}
   \end{definition}
When $\Gamma $ is Abelian, Blei  \cite{blei75dio} proves that every set $E$ that is convergently embedded in some compact Abelian group $K$ is a \emph{Rosenthal set}, i.e., every essentially bounded function on $G$ with Fourier transform supported on $E$ is equal almost everywhere to a continuous function (that is, $L^\infty_E(G)=C_E(G)$), see \cite[Theorem B]{blei75dio}. Rosenthal sets are known to be Riesz sets, \cite{drespigno74,li93}. We now set to extend this to the Fourier algebra of non-Abelian groups and  see that such sets can be found in a wide class of groups.

   \begin{theorem}\label{ceimplRiesz}
  Let $\Gamma$ be  a discrete group. If $E\subset \Gamma$ is convergently embedded in $K$ for some locally compact group $K$, then $E$ is a Riesz set.
   \end{theorem}
  \begin{proof}
  We will let $p$ denote the only limit point of $\alpha(E)$ in $K$.

   Assume $E$ is not Riesz. By Lemma \ref{point}, there is $x_0\in E$ and  $\phi\in B_E(\Gamma)$ such that $\phi(x_0)=0$ and  $\phi_s(x_0)\neq 0$. Let $U\subseteq K$ be a
    neighbourhood of $p$ with $\alpha(x_0)\notin U$.

    By Lemma 3.1 of \cite{lauwong13}, there are, for each $n\in \N$, a neighbourhood $V_n$ of $p$ and
    $u_n\in A(K)$ such that
   \[ \restr{u_n}{V_n}=1, \quad  \restr{u_n}{G\setminus U}=0  \quad \mbox{ and } \quad \norm{u_n}\leq 1+\frac1n.\]
   Since  $\overline{\alpha(E)}$ is compact and $p$ is the only limit point of $\alpha(E)$, $\alpha(E)\setminus V_n$ has to be finite. Let $F_n$ be the inverse image of this set under $\alpha$.  $F_n$ is therefore a finite subset of $E$ with $\alpha(F_n)=\alpha(E)\setminus V_n$. Define next
   \[\phi_n:= \phi-\sum_{x\in F_n} \phi(x)\delta_x\in B(\Gamma).\]
   Then $\supp \phi_n \subseteq E\cap  \alpha^{-1}(V_n)$.  Define now $\underline{u_n}$ to be $ u_n\circ \alpha$ on $\<{E}$ and 0 elsewhere. Then $\underline{u_n}\in  B(\Gamma)$ , and  we have
   \begin{equation}\label{0} \phi_n  (1-\underline{u_n})=0.\end{equation}
   Observe as well that, by  \cite[ Theorem 2.2.1 (ii)]{kanilaubook}
   \[ \norm{\underline{u_n}}=\norm{u_n}\leq 1+\frac1n.\]

   Let for each $n$, $\phi_n=\left(\phi_n\right)_a+\left(\phi_n\right)_s$ be its Lebesgue decomposition. Since $\sum_{x\in F_n} \phi(x)\delta_x\in A(\Gamma)$, we know that
   $\left(\phi_n\right)_s=\phi_s$.

   From \eqref{0}, we deduce then that
  \begin{equation}\label{ag}\phi_s(1- \underline{u_n})=-\left(\phi_n\right)_a (1- \underline{u_n})\in A(\Gamma).\end{equation}
  We then deduce  that, since $\phi_s  \underline{u_n}=\phi_s-\phi_s(1-\underline{u_n})$, the first summand being singular and the second being in $A(\Gamma)$,
  %{\tt Do we need an approximate identity in $A(\Gamma)$??}
  \[\norm{\phi_s \underline{u_n}}= \norm{\phi_s}+\norm{\phi_s(1-\underline{u_n})}.\]
  Thence,
  \begin{align*}
    \norm{\phi_s  (1-\underline{u_n})}&=\norm{\phi_s  \underline{u_n}}- \norm{\phi_s}\\
    &\leq \norm{\phi_s}  \left(1+\frac1n\right)-\norm{\phi_s}\\&=\frac{1}{n}\norm{\phi_s}.
  \end{align*}
  Now,
   \[\norm{\phi_s  \left(1-\underline{u_n}\right)}\geq  \norm{\phi_s(1-\underline{u_n})}_\infty\geq \left|  \phi_s(x_0)\right|.\]
   By taking $n$ large enough so that $\left|\phi_s(x_0)\right|>\dfrac{1}{n}\norm{\phi_s}$,    the desired contradiction is reached.
   %There are  then $\phi\in A_E(G)$ and $x_0\in \Gamma$ such that
  \end{proof}
  We will need the following result that follows from well-known estimates.

  \begin{lemma}
    \label{cyclic}
    Let $\Gamma$ be a cyclic group of order $N$ ($1\leq N\leq \infty$) with character group  $G$ and  let $\gamma \in \Gamma$ be its generator.
Consider for each $n\in \N$, $u_n=\sum_{j=0}^{n-1} \delta_{\gamma^j}\in A(\Gamma)$ and its Fourier transform $\widehat{u_n}\colon G \to \T$. For $0<p<1$, the following estimates hold.
\begin{align*}
&(1)\mbox{ If $N=\infty$, } \mbox{there are } C,D_p >0 \mbox{ such that for each } n\in\N, \\&\norm{\widehat{u_n}}_{L^1(G)}\geq C \log n   \mbox{ and }  \norm{\widehat{u_n}}_{L^p(G)}\leq  D_p,  \\
&(2)\mbox{ If  $N<\infty$, }\, \; \norm{\widehat{u_N}}_{L^1(G)}=1 \quad \mbox{ and }\quad \norm{\widehat{u_N}}_{L^p(G)}=N^{1-\frac{1}{p}}.
\end{align*}
%\begin{align*}
%(1)& \lim_{n\to \infty} \frac{\norm{\widehat{u_n}}_{L^p(G)}}{\norm{\widehat{u_n}}_{L^1(G)}}=0, \mbox{ if $G=\T$ }, \\
%(2)&\,  \norm{\widehat{u_N}}_{L^1(G)}=1 \quad \mbox{ and }\quad \norm{\widehat{u_N}}_{L^p(G)}=N^{1-\frac{1}{p}}, \mbox{ if $G=\T_N$}.
%\end{align*}%
%\begin{itemize}
%            \item[--]    If the order of $\Gamma$ is infinite, \[ \lim_{n\to \infty}  \frac{\norm{\widehat{u_n}}_{L^p(G)}}{\norm{\widehat{u_n}}_{L^1(G)}}=0.\]
%            \item[--]  If $N$ is finite,
%            \begin{align*}
%\norm{\widehat{u_N}}_{L^1(G)}&=1, \mbox{ and }\\
%\norm{\widehat{u_N}}_{L^p(G)}&=N^{1-\frac{1}{p}}.
    %\end{align*}
          %\end{itemize}
    %\begin{itemize}
%       \item If the order of $\Gamma$ is infinite, then there are constants $C$ and $D$ such that \begin{align*}
%\norm{\widehat{u_n}}_{L^1(G)}&\geq C \log n, \mbox{ and }\\
%\norm{\widehat{u_n}}_{L^p(G)}&\leq D\left(\frac{1}{1-p}\right)^{\frac{1}{p}}, \mbox{ for any } 0<p<1.
%    \end{align*}
%\item If $\Gamma$ has  finite order $N$
%\begin{align*}
%\norm{\widehat{u_N}}_{L^1(G)}&=1, \mbox{ and }\\
%\norm{\widehat{u_N}}_{L^p(G)}&N^{1-\frac{1}{p}}.
%    \end{align*}
%    \end{itemize}
      \end{lemma}
			
	\begin{proof}
We represent $G$ as  the circle group  $\T$ if the order of $\gamma$ is infinite and as $\T_N=\{e^{2\pi i j/N}\colon 1\leq j\leq N\}$ if the order of $\gamma$ is $N<\infty $. Then, in both cases,
\[\left|\widehat{u_n}(e^{2\pi it}) \right|=\left|\sum_{j=1}^n e^{2\pi {i j
t}}\right|=\left|\frac{\sin n\pi t}{\sin \pi t}\right|.\]

%t}}\right|=\left|\frac{\sin \left(2n+1\right)\pi t}{\sin \pi t}\right|.\]

In the case of infinite order, the easy part of a well-known estimate  (see  Exercise II.1.1 in \cite{katz}) can be used to  show that, form some $C>0$,
\begin{equation}\label{cy1}\norm{\widehat{u_n}}_{L^1(\T)}\geq C \log n.\end{equation}
On the other hand, noting that
\[\left|\frac{1}{\sin (\pi t)}\right|\leq \frac{2}{\pi t}, \mbox{ for all $0<t<1$},\]
one has that, for some constant $D>0$,
\begin{equation}\label{cy2}\norm{\widehat{u_n}}_{L^p(G)}\leq D\left(\frac{1}{1-p}\right)^{\frac{1}{p}}.
\end{equation}
When $\Gamma$ is finite, $u_N=\mathbf{1}$,   the constant 1-function and  $\widehat{u_N}=N \delta_1$ (recall that although   $\Gamma$ and $G$ are isomorphic, their Haar measures are different, (normalized  for $G$ and counting measure for $\Gamma$). Hence $\norm{\widehat{u_N}}_{L^1(G)}=1$  and  $
\norm{\widehat{u_N}}_{L^p(G)} =N^{1-\frac{1}{p}}$.
  \end{proof}

  Reflexivity of subspaces of $L^1(G)$ can be characterized through  uniform integrability of its unit ball, see e.g. \cite[Theorem III.C.12]{wojtbook}. This was used in \cite{hare88} to show that $L_E^1(G) $ is not reflexive if $E$ is not a $\Lambda(1)$-set.  In the case of  preduals of von Neumann algebras, such as  $A(\Gamma)$, one can use Lemma \ref{takewc} below as a substitute. This will require some notation.
    Recall  that, in  a von Neumann algebra $\mathcal{M}$, every self adjoint   linear functional $\phi \in \mathcal{M}^\ast$,  can be decomposed as $\phi =(\phi^+ - \phi^-)$, with $\phi^+,\, \phi^-$ positive functionals. If $\phi=V|\phi|$    denotes the polar decomposition of $\phi$,  then
\[ \phi^+=\frac{1}{2}\left(|\phi|+\phi\right), \quad \phi^-=\frac{1}{2}\left(|\phi|-\phi\right)).\]
A general $\phi\in  \mathcal{M}^\ast$ can then be written as $\phi=\phi_1+i \phi_2$ with
$\phi_1=\dfrac{1}{2}(\phi+\phi^\ast)$ and $\phi_2=\dfrac{i}{2}(\phi^\ast-\phi)$, and both $\phi_1$ and $\phi_2$ are self adjoint.
 We denote then
\[ [\phi]=\phi_1^++\phi_1^-+\phi_2^++\phi_2^-.\]

\begin{lemma}[Lemma 5.6 of \cite{take02}]\label{takewc}
Let $\mathcal{M}$ be a von Neumann algebra with predual $\mathcal{M}_\ast$. If $K\subset \mathcal{M}_\ast$ is relatively weakly compact, then for any $\varepsilon>0$, there exist a finite set $F_\varepsilon \subseteq K$ and $\delta>0$ such that if $p$ is a projection in $\mathcal{M}$ and
$\<{[\phi],p}<\delta$ for every $\phi \in F_\varepsilon$, then $\left|\<{\phi,p}\right|<\varepsilon$ for every $\phi \in K$.
\end{lemma}

\begin{lemma}\label{nonref}
  Let $\Gamma$ be a discrete group. Let $\{\chi_n\colon n\in \N\}$ be a subset of $\Gamma$  such that, for each $n\in\N$, the order
	$o(\chi_n)$ of $\chi_n$ is   larger than $n $. Put $k(n)=o(\chi_n)$ if the order of $\chi_{n}$ is finite  and $k(n)=n$, otherwise, and define
  \begin{align*}
    F_n&=\left\{ \chi_n^j \colon 1\leq j\leq k(n) \right\} \quad \mbox{ and  } \quad
    E=\bigcup_{n\in \N } F_n.
  \end{align*}
  Then $A_E(\Gamma)$ is not reflexive.
\end{lemma}

\begin{proof}
  Consider the functions
  \[ u_n=B_n \sum_{j=0}^{k(n)-1} \delta_{\chi_n^j}   \in A(\Gamma),\]
  where $B_n$ is chosen so that $\norm{u_n}_{A(\Gamma)}=1$. If $G_n$ is   the character group of the cyclic subgroup $\Gamma_n\subseteq \Gamma$ generated by $\chi_n$,  then by  Proposition 2.4.1 of \cite{kanilaubook}, one has that $\norm{\widehat{u_n}}_{L^1(G_n)}=\norm{u_n}_{A(\Gamma_n)}=\norm{u_n}_{A(\Gamma)}=1$.

Put $\alpha_n=\norm{\widehat{u_n}}_{L^{1/2}(G_n})$. By Lemma \ref{cyclic},
 \[\lim_{n\to \infty}  \frac{\norm{\widehat{u_n}}_{L^p(G)}}{\norm{\widehat{u_n}}_{L^1(G)}}=0,\]
hence
 $ \lim_{n\to \infty}\alpha_n=0.$

%After inequalities \eqref{cy1} and \eqref{cy2}, it is clear that  \[\lim_{n\to \infty}  \frac{\norm{\widehat{u_n}}_{L^p(G)}}{\norm{\widehat{u_n}}_{L^1(G)}}=0.\]

We now define  %\[\norm{\widehat{u_n}}_{L^{1/2}(G_n)}\leq \frac{C}{\log n}. \]
  \[ \mathcal{A}_n=\left\{x \in G_n \colon |\widehat{u_n}(x)|>\frac{1}{2\alpha_n}\right\}.\]
  Then, letting $\m{$G_{n}$}$ denote the (normalized) Haar measure of $G_n$,  \begin{equation}
    \label{man}
  \m{$G_n$}(\mathcal{A}_n)\leq 2\alpha_n.\end{equation} On the other hand, since
  $\left|\widehat{u_n}(x)\right|\leq \left(\frac{1}{2\alpha_n}\right)^{1/2}  \left|\widehat{u_n}(x)\right|^{1/2}$, for every $x\in \mathcal{A}_n^c$, we have that
  \begin{align*}
\alpha_n^{\frac{1}{2}}&= \norm{\widehat{u_n}}_{L^{1/2}(G_n)}^{\frac{1}{2}}\geq \int_{\mathcal{A}_n^c} |\widehat{u_n}(x)|^{1/2} d\m{$G_n$}(x)\geq\left(2\alpha_n\right)^{1/2}\int_{\mathcal{A}_n^c} |\widehat{u_n}(x)| d\m{$G_n$}(x).
  \end{align*}
  It follows that
  \[\int_{\mathcal{A}_n^c} |\widehat{u_n}(x)| d\m{$G_n$}(x)\leq \sqrt{\frac{1}{2}}\]
    %\left(  \frac{1}{2}\right)^{\frac{1}{2}}
  and, hence, that
\begin{align}    \label{fnan}
 \int_{\mathcal{A}_n} |\widehat{u_n}(x)|d\m{$G_n$}(x)&\geq 1- \sqrt{ \frac{1}{2}}.
  \end{align}

Let now  $r_n\colon A(\Gamma)\to A(\Gamma_n)$ be the restriction mapping (see \cite[Proposition 2.6.6]{kanilaubook})  and define $p_n\in VN(\Gamma)$ as $p_n=r_n^\ast\left(\widehat{\Cf{\mathcal{A}_n}}\right)$, with $r_n^\ast\colon VN(\Gamma_n)\to VN(\Gamma)$ being the adjoint of $r_n$ and $\widehat{\Cf{\mathcal{A}_n}}$ the Fourier transform of $\Cf{\mathcal{A}_n}\in L^\infty(G_n)$.

 % For each $n$ we consider now the projection $p_n\in VN(\Gamma)$ that is obtained after applying to $\widehat{\Cf{\mathcal{A}_n}}\in VN(\Gamma_n)$ the adjoint of the restriction mapping $r\colon A(\Gamma)\to A(\Gamma_n)$ (see \cite[Proposition 2.6.6]{kanilaubook}).
 Observe that
  \begin{equation}
    \label{pndelta}
  \<{\delta_x,p_n}=0,\quad \mbox{ if } x\notin \Gamma_n.\end{equation}

  Suppose now that $A_E(\Gamma)$ is reflexive, then the set $\{u_n\colon n\in \N\}$ is relatively weakly compact. Pick $0<\varepsilon<1-\sqrt{\frac{1}{2}}$. By Lemma \ref{takewc}, there are then  $n_1,\ldots,n_k$ and $\delta>0$ such that, for any projection $p$ in $VN(\Gamma)$,
  \begin{equation}
    \label{wc}
 \<{[u_{n_i}],p}<\delta, \; i=1,\ldots k,\quad  \implies    \left| \<{u_{n},p}\right|<\varepsilon\mbox{ for every $n\in \N$}.  \end{equation}
Pick, for each  $i=1,\ldots,k$, a function $v_i \in A(\Gamma)$ with finite support $\mathcal{F}_i\subseteq \Gamma$
such that
\[ \norm{v_i-[u_{n_i}]}<\frac{\delta}{2}\]
   Next, for each $n\in \N$ and $i=1,\ldots,k$ we  define  $w_{i,n}\in A(\Gamma)$
as $v_{i}$ on $\Gamma_n$ and 0  elsewhere. Then, by  \eqref{pndelta},
\begin{align*}
\<{v_{i},p_n}&=\<{w_{i,n},p_n}.
 \end{align*}
%
%with support $\mathcal{F}_{i,n}$ such that
%
% \mathcal{F}_{i,n}&= \supp v_i\cap \Gamma_n\\
%\mathcal{F}_{i,n}=\mathcal{F}_{i,m}& \mbox{ implies } w_{i,n}=w_{i,m}\\

Since the support of each $v_i$ is finite,  we will be able to find  $w_1,\ldots,w_k$  and $n_0\in \mathbb{N}$ such that
 \begin{align*}
 w_{i,n}&=w_i, \mbox{ for every $i=1,\ldots,k $ and every $n\geq n_0$}.
 \end{align*}
  We finally observe that, putting $D=\max_i \norm{w_i}_{\ell^1(\Gamma)}$,
\[\left|\<{w_i,p_n}\right|=\left|\int_{\mathcal{A}_n} \widehat{w_i}(x)d\m{$G_n$}(x)\right|\leq D \, \m{$G_n$}(\mathcal{A}_n), \]
for each $n\geq n_0$ and each $i=1,\ldots,k$. Since, \eqref{man},  $ \m{$G_n$}(\mathcal{A}_n)\leq 2\alpha_n$, and $\lim_n \alpha_n=0$,  we deduce that
for $n\geq n_0$ large enough, $\left|\<{w_i,p_{n}}\right|<\delta/2$ ($i=1,\ldots,k$) and, hence, for such $n$ and any $i=1,\ldots,k$,
\begin{align*}
  \<{[u_{n_i}],p_{n}}&\leq \left|  \<{[u_{n_i}]-v_{i},p_{n}}\right|+
  \left|\<{w_i,p_{n}}\right|\\
  &\leq \norm{[u_{n_i}]-v_{i}}\cdot \norm{p_{n}}+ \left|\<{w_i,p_{n}}\right|=\delta.
\end{align*}
The choice of $u_{n_1}$, \ldots, $u_{n_k}$ allows us to conclude   (by condition \eqref{wc}), that if $A_E(\Gamma)$ were reflexive, then one would have  that,  given  $n\geq n_0$ large enough,
\[  \<{u_{m},p_{n}}<\varepsilon, \mbox{ for every $m\in \N$}.\]
But we know, \eqref{fnan},  that for every $n\in \N$,
\[ \left|\<{u_{n},p_{n}}\right|> 1-\sqrt{\frac{1}{2}}> \varepsilon,\]
showing that  $A_E(\Gamma)$ cannot be reflexive.
\end{proof}
\begin{theorem}
  \label{gen:riesznotlp}
  Let $\Gamma$ be a discrete group  containing a countable subset  $A$ that is  convergently embedded  in a locally compact, metrizable group $K$ and such that $\<{A}$ has only finitely many elements of each order. Then $\Gamma$ contains a   Riesz set $E$   such that $A_E(\Gamma)$ is not reflexive.
\end{theorem}
\begin{proof}
  Whenever necessary, we will consider in this proof a   left-invariant distance $d$ on $K$  that induces its topology.

Let $\alpha\colon \<{A} \to K$ be  the isomorphism provided by Definition \ref{def:ce} and let $k_0$ denote the only limit point of $\alpha(A)$. We can
 find a sequence by $(\psi_n)_n\subseteq A$ with $\psi_{n+1}\notin \{\psi_1,\ldots,\psi_n\}$ for each $n\in \N$ such that   $\lim_n \alpha(\psi_n)=k_0$.

  We then define a new sequence
  \[ \chi_n=\psi_{n+1}\cdot  \psi_n^{-1}.\]
  The image under $\alpha$ of this new sequence will converge to the identity 1 of $K$.
  Since  $\<{A}$ has only finitely many elements of each order, we can  assume, upon removing some terms,  that
  $o(\chi_n)>n$ for every $n\in \N$.
  %We don't exclude here the possibility that $\alpha(\chi_n)=k_0$ for every $n\in \N$.

  Find, for each $n\in \N$, $N(n)>n$ large enough so that
  \begin{equation}
    \label{N(n)} \di{\alpha(\chi_{\txt{\tiny$N(n)$}})}{1} <\frac{1}{n^2}. \end{equation}

  As in Lemma \ref{nonref}, we    put $k(n)=o(\chi_{\txt{\tiny$N(n)$}})$ if the order of $\chi_{\txt{\tiny$N(n)$}}$ is finite and $k(n)=N(n)$ otherwise. Next, we define
  \begin{equation}
    \label{fns}
    F_n=\set{\chi_{\txt{\tiny$N(n)$}}^{j}}{1\leq j\leq k(n)}
  \end{equation}
  and
  \begin{equation}
    \label{E}
   E=\bigcup_{n\in \N} F_n.
  \end{equation}
  We next  see that $E$ is a Riesz set. This will finish the proof, since $A_E(\Gamma)$ cannot be reflexive by Lemma \ref{nonref}.

   By Theorem  \ref{ceimplRiesz}, it is enough to check that, after ordering the set $N=\{(n,j)\colon 1\leq j\leq n,\;n\in\N\}$ lexicographically,
    the sequence $\{\alpha\left(\chi_{N(n)}^j\right)\colon (n,j)\in N\}$ converges     to $1$.  Let to that end $\varepsilon>0$.

   Pick $n_0$ such that $\dfrac{1}{n_0}<\varepsilon$ and take any $n\geq n_0$ and $1\leq j\leq n$.
   Then,
\begin{align*}
  \di{\alpha(\chi_{\txt{\tiny$N(n)$}}^{j})}{1}&\leq \\&\leq \di{\alpha(\chi_{\txt{\tiny$N(n)$}}^{j})}{\alpha(\chi_{\txt{\tiny$N(n)$}}^{j-1})}+
  \di{\alpha(\chi_{\txt{\tiny$N(n)$}}^{j-1})}{\alpha(\chi_{\txt{\tiny$N(n)$}}^{j-2})}\ldots + \di{\alpha(\chi_{\txt{\tiny$N(n)$}})}{1}\\&= j \cdot \di{\alpha(\chi_{\txt{\tiny$N(n)$}})}{1}\leq \frac{1}{n}<\varepsilon.
\end{align*}%
%   \label{chitoe}
%     \di{\alpha(\chi_{\txt{\tiny$N(n)$}}^{j})}{k_0^{n!j}}\leq \textcolor{red}{  j     \, \di{\alpha(\chi_{\txt{\tiny$N(n)$}})}{k_0}\overset{\eqref{N(n)}}{\leq }} \frac{n!\sqrt{n}}{n!n}<\frac{\varepsilon}{2}.
%   \end{equation}%
%Inequality  \eqref{chitoe} together with the choice of $n_0$, shows that
%\[ \di{\chi_{\txt{\tiny$N(n)$}}^{n!j}}{1}\leq \di{\chi_{\txt{\tiny$N(n)$}}^{n!j}}{k_0^{n!j}}+\di{k_0^{n!j}}{1}< \varepsilon,\]
%for every $n\geq n_0$ and every $1\leq j\leq \sqrt{n}$. This proves that $1$ is the only limit  point of $E$ in $\T$. Lemma \ref{bleib} yields  then that $E$ is a Rosenthal set for the group $G/\Lambda^\perp$ whose character group is $\Lambda$.
%% Lemma \ref{quot} shows then that $E$ is actually a Rosenthal set for $G$.
\end{proof}

It follows quite easily from Theorem  \ref{gen:riesznotlp} that every  Abelian group $\Gamma$ contains $E\subseteq \Gamma$  that is Riesz but  produces a nonreflexive Arens regular ideal $A_E(\Gamma)=L^1_E(G)$, where $G=\widehat \Gamma$, see Remark \ref{remabref}. We choose to prove a stronger theorem that might be of independent interest. This will require a little effort. We start with the following Lemma.

\begin{lemma}[Corollary 2.3 of \cite{hare88pac}]
  \label{hare88pac} Let $\gm$  be a discrete group and $E\subseteq \gm$ be a $\Lambda(p)$-set for some $p>0$.  There are then  $c$ and $0<\varepsilon<1$ such that for any $\chi_1,\ldots,\chi_N\in E$
  \[\left|\left\{ \prod_{j=1}^N \chi_j^{k_j}     \colon k_j\in\{0,1\}\,\right\}\right|\leq c 2^{\varepsilon N}.\]
\end{lemma}

 \begin{corollary}\label{cor:rnotlp:ab}
Every discrete Abelian group $\Gamma$, contains a subset $E$ that is  Rosenthal but is not $\Lambda(p)$ for any $p>0$.\end{corollary}

\begin{proof}
Suppose first that $\Gamma$ contains an infinite independent subset $A=\{\chi_n \colon n\in \N\}$. Partition $A=\bigcup A_n$ with $|A_n|=n$ for each $n\in \N$. Enumerate $A_n=\{\chi_{n,j}\colon 1\leq j\leq n\}$.

  Define, for each $n\in \N$,
  \[F_n=\left\{ \prod_{j=1}^n \chi_{n,j}^{k_j}\colon k_j\in \{0,1\}\right\}.\]
  Then, for $n\neq m\in \N$, $\<{F_n}\cap \<{F_m}=\{1\}$. It follows directly from Lemma 2.3 of \cite{blei75dio} that  $E:=\bigcup_{n\in \N}F_n$  is a Rosenthal set.

  On the other hand, the independence of $A_n$ implies that $\left|F_n\right|=2^n-1$ and  Lemma \ref{hare88pac} shows that $E$ cannot be a $\Lambda(p)$ set.

  If $\Gamma $ does not contain any infinite independent set,  then either $\Gamma$ contains  a copy of the integers $\Z$ or
   there are primes $p_1,\ldots,p_N$ and $q_1,\ldots,q_M$ (not necessarily different) such that
\[\Gamma\cong \left(\bigoplus _{j=1}^N \Z(p_j)\right)\bigoplus  \left(\bigoplus _{j=1}^M \Z(q_j^\infty)\right),\]
see \cite[4.1.1, 4.1.3]{robi96}.
In both cases, $\Gamma$ contains a countably infinite   subgroup $\Lambda$ with at most finitely many elements of the same finite order. If one enumerates $\Lambda=\{\chi_n\colon n\in \N\}$ and picks for each $n\in \N$,  an element $x_{n}\in G$ with $\chi_{n}(x_n)\neq 1$, then, naming $S$ the subgroup generated by $\{x_n \colon n\in \N\}$,  we obtain that  the evaluation map
$\alpha \colon \Lambda\to \widehat{S_d}$ is a group isomorphism with dense range. We then take a sequence $A$ in $\Lambda$ such that $\alpha(A)$ converges to some $p\in \widehat{S_d}\setminus \alpha(\Lambda)$. This set $A$ is then convergently embedded in $\widehat{S_d}$.  Theorem  \ref{gen:riesznotlp} proves that $\Gamma$ contains a Riesz set $E$ such that $A_E(\Gamma)$ is not reflexive.   Theorem  B of \cite{blei75dio} proves that  $E$, as every  subset of an Abelian group that is convergently embedded in a compact group, is even a  Rosenthal set.
%
%
%For each $n\in \N$, pick $M_n\in \N$ such that $2^{M_n-1}\leq n\leq 2^{M_n}$, then, using the binary expansion of integer numbers $1\leq j\leq n$, we see that
% \[ F_n\subseteq \left\{ \prod_{i=1}^{M_n-1} \alpha  \left(\chi_{N(n)}\right)^{k_i 2^i}\colon k_i \in \{0,1\}\right\}.\]
% By Lemma \ref{hare88pac}, if $E$ is a $\Lambda(p)$-set for some $p>0$,  there are constants $c$ and $0<\varepsilon<1$ such that, for every $n\in \N$,
% \[\left|F_n\right|\leq c 2^{(M_n-1)\varepsilon}.\]
%On the other hand, the cardinality of $F_n$ is $n$ (the order of $\chi_{N(n)}$ is larger than $n$).
%So, if  $E$ were a $\Lambda(p)$-set for some $p>0$,
%\[n= \left|F_n\right| \leq c 2^{(M_n-1)\varepsilon}\leq c n^{\,\varepsilon}, \mbox{ for every $n\in \N$}.\]
% This contradiction shows that $E$ cannot be  $\Lambda(p)$-set for any $p>0$.
\end{proof}
\begin{remarks} \label{remabref}
We gather here some remarks that seem worthwhile to be brought forward.
\begin{enumerate}
  \item Every countable maximally almost periodic  discrete group is convergently embedded in some compact metrizable group,  the proof is very similar to that of Corollary \ref{cor:rnotlp:ab}.
  \item
To the best of our knowledge, Corollary \ref{cor:rnotlp:ab} is the first proof of the existence, in an arbitrary  Abelian group,  of Rosenthal sets that are not $\Lambda(p)$ for any $p>0$. It certainly  does not follow from the results of Blei in \cite{blei75dio}, see Remark (b) in page 197 of Blei [loc. cit.].

\item The proof of Theorem   \ref{gen:riesznotlp}
 can be greatly simplified if $\Gamma$ is Abelian.    Once the set $E$ is defined, one can apply Lemma 4.1 of \cite{rudin60} to see that $E$ is not a $\Lambda(1)$-set, so that  Lemma \ref{nonref} is not necessary. By the Corollary of \cite{hare88},  $A_E(\Gamma)$ is then nonreflexive.
\end{enumerate}
    \end{remarks}
%
%\begin{remark}
%\end{remark}
%\begin{remark}\end{remark}

For the next Corollary we recall that a (metrizable) topological group is locally precompact if it  can be embedded in a locally compact (metrizable) group.

\begin{corollary}
  If an amenable group $\Gamma$ contains a subgroup $\Lambda$ with finitely many elements of each order and $\Gamma$  admits a nondiscrete metrizable locally pre-compact topologgy, then $A(\Gamma)$  contains Arens-regular ideals that are not reflexive.\end{corollary}

\subsection{Strongly Arens irregular ideals, with and without a bai}

In a \wasabi algebra, a closed ideal with a bai stays in the \wasabi class, and so it is sAir.
In particular, as already deduced in Corollary \ref{cor:SAIBAI}, for a compact group $G$ and a discrete amenable group $\Gamma$,
the ideals $L^1_E(G)$ and $A_E(\Gamma)$ are sAir whenever, respectively,  $E\in\Omega(\widehat{G})$ and $E\in \Omega(\Gamma)$.
The natural question of whether ideals $L_E^1(G)$ and $A_E(\Gamma)$ not issued from the (hyper-)coset ring may be sAir seems to be in order.
The following general construction  yields such ideals.

%It is not easy to produce natural examples of sAir ideals in $L^1(G)$ or $A(\Gamma)$ which do not have the form $L_E^1(G)$, $A_E(\Gamma)$, with $E$ in the (hyper-)coset ring. The following is a general construction that yields such examples.

\begin{theorem}\label{gen:bairef} Let $ \mA $ be a Banach algebra Let $J_B$ and $J_1$ be closed ideals of $\mA$, such that $J_B$ has a  bai and is sAir.  If $J_B$ and $J_1$ are orthogonal, i.e.,  $J_B J_1 = J_1 J_B =\{0\} $, then
the internal direct sum  $J_0= J_B \oplus  J_1$ is a closed ideal and
\[\mathcal Z_i(J_0^\bd)= \mathcal Z_i(J_B^\bd)+\mathcal Z_i(J_1^\bd),\quad i=1,2.\]
%that is  SAI.
\end{theorem}

\begin{proof}
  We first see that    $J_0$ is a  closed ideal. Let $e_B$ be the right identity in $J_B^{**}$ associated with the bai in $J_B.$
	If $(u_n)_n$ is a sequence contained in $J_B$ and $(v_n)_n$ is a sequence contained in $J_1$ such that $\lim_n u_n+v_n=a\in \mA$, then $e_Bu_n=u_n$ and $e_Bv_n=0$ for every $n,$ and so
	\[\lim_n u_n=\lim_n e_B   u_n=\lim_n e_B(u_n+v_n)=e_B a \in J_B.\]
	%(\sigma(\mA^\bd,\mA^\ast)-\text{convergence}).\]
	%\textcolor{red}{Why?}
	
	This automatically implies that $\lim_n v_n=a-e_B  a$ and hence that $a\in J_0$.

  We next see that $J_0^\bd=J_B^\bd \oplus J_1^\bd$. This will follow at once if we see that
  \begin{equation}
    \label{normbd}
    \norm{a+b}\geq \norm{a}, \mbox{ whenever $a\in J_B$ and $b\in J_1$}.
  \end{equation}
  To see this, let $\varepsilon>0$ and let $\phi \in \mA^\ast$, $\norm{\phi}\leq 1$ be such that
  \[ \norm{a}\leq \left|\<{a,\phi}\right|+\varepsilon.\]
  Since $\<{a,e_B\cdot \phi}=\<{a,\phi}$ and $\<{b,e_B\cdot\phi}=0$, one has that
  \[\norm{a+b}\geq \left|\<{a+b,e_B\cdot \phi}\right|=\left|\<{a,\phi}\right|\geq \norm{a}-\varepsilon,\]
 and inequality \eqref{normbd} is satisfied.
 %Hence,  for any bounded   net of the form $(a_\alpha+b_\alpha)_\alpha$, $a_\alpha \in J_B^\bd$, $b_\alpha\in J_1^\bd$, one has that %$(a_\alpha)_\alpha$ is bounded as well.

Let now $i=1$. The proof is the same when $i=2.$
Let $m,q\in  J_0^\bd$. By the above, we write $m=m_B+m_1$ and  $q=q_B+q_1$ with $m_B, q_B\in J_B^\bd$ and $m_1, q_1\in J_1^\bd$.
Then, \[m_1  q_B=m_1\lozenge q_B=m_Bq_1=m_B\lozenge q_1=0.\] So,\begin{equation}\label{centre}\begin{split}
m   q&=(m_B+m_1)  (q_B+q_1) =m_B  q_B+m_1q_1\mbox{ and }\\m \lozenge q &=(m_B+m_1)\lozenge (q_B+q_1) =m_B\lozenge q_B+m_1\lozenge q_1.\end{split}
\end{equation}

%\begin{align*}m   q&=(m_B+m_1)  q =m_B  q\mbox{ and }\\m \lozenge q &=(m_B+m_1)\lozenge q =m_B\lozenge q.\end{align*}
%Therefore, $m  q=m\lozenge q$ if and only if $m_B  q=m_B\lozenge q.$

 Accordingly, if $m\in  \mathcal Z_1(J_0^\bd)$,
then taking $q_1=0$ in (\ref{centre}), we find $m_B\in \mathcal Z_1(J_B^\bd)$. Taking $q_B=0$ in (\ref{centre}), we find
$m_1\in \mathcal Z_1(J_1^\bd).$ Thus,
$\mathcal Z_1(J_0^\bd)\subseteq \mathcal Z_1(J_B^\bd)+\mathcal Z_1(J_1^\bd).$
The reverse inclusion is also quick as it is clear that $mq=m\lozenge q$ if  $m_B q_B=m_B\lozenge q_B$ and $ m_1 q_1=m_1\lozenge q_1.$
\end{proof}

\begin{corollary}
  \label{cor:sum}
  Let $ \mA $ be a Banach algebra in \wasabi. Let $J_B$ and $J_1$ be closed ideals of $\mA$ with $J_B J_1=J_1 J_B=\{0\}$, such that $J_B$ has a  bai. Let $J_0=J_B\oplus J_1$.
   \begin{enumerate}
     \item If $J_1$ is reflexive, then $J_0$ is sAir.
     \item If $J_1$ is Arens regular but not reflexive, then $J_0$ is neither Arens regular nor sAir
   \end{enumerate}
   \end{corollary}

   \begin{proof}
     Assume first that $J_1$ is reflexive and let $m\in \mathcal Z_1(J_0^\bd)$. By Theorem \ref{gen:bairef}, $m=u+m_1$ with $u\in J_B$ and $m_1\in J_1^\bd$. Since  $J_1$  is reflexive, it follows right away that $m\in J_0$.

     Assume now that $J_1$ is Arens regular but not reflexive. Let $m_1\in J_1^\bd \setminus J_1$. Then, given $q=q_B+q_1\in J_0^\bd$, $q_B\in J_B^\bd$, $q_1\in J_1^\bd$,
     \begin{align*}
       m_1  q&=m_1  q_1 \mbox{ and } \\
m_1\lozenge q&=m_1\lozenge q_1.
     \end{align*}
     Since $J_1$ is Arens regular, we see that $m_1\in \mathcal Z_1(J_0^\bd)$ and so $J_0 $ is not sAir.

     On the other hand,  Theorem \ref{gen:bairef} shows that no  $m\in J_B^\bd\setminus J_B$ can be in
		$\mathcal Z_1(J_0^\bd)$, whence the Arens irregularity of $J_0$.
   \end{proof}
	
We can now  find examples of sAir closed ideals  that have no bai in many Fourier algebras.

\begin{example}
Let $\Gamma $ be a discrete amenable group. Assume that $\Gamma $  contains   a proper  infinite subgroup $\Lambda$.
Then $A(\Gamma)$ contains an ideal that is sAir but contains no bai.
\end{example}

\begin{proof}
Since $ \Lambda$ is infinite, $\Gamma\setminus \Lambda$ is infinite. By \cite[Theorem 1]{pica73}, we can find  $F\subseteq \Gamma\setminus \Lambda$ which  is an infinite $\Lambda(4)$-set, $A_F(\Gamma)$ is then reflexive by Lemma \ref{lambda1=ref,ag}. If $E=\Lambda\cup F$, Corollary \ref{cor:sum} shows that $A_E(\Gamma)=A_\Lambda(\Gamma)\oplus A_F(\Gamma)$ is sAir.
   If $A_E(\Gamma)$ had a bai, then $E\in \Omega(\Gamma)$ and so $F\in \Omega(\Gamma)$ so that  $A_F(\Gamma)$ has a bai. But this latter ideal is reflexive, so $A_F(\Gamma)$ should have  an identity, i.e., $\Cf{F}\in A_F(\Gamma)$ which is impossible, for  $A_F(\Gamma) \subseteq c_0(\Gamma)$.
\end{proof}
\begin{example}
  Let $\Gamma$ be an amenable group that contains a proper  infinite subgroup $\Lambda$. If $\Lambda$ is either Abelian or satisfies the hypothesis of Theorem \ref{gen:riesznotlp}, then $A(\Gamma)$ contains ideals that are neither Arens regular not sAir.
\end{example}
\begin{proof}
  We use Theorem \ref{gen:riesznotlp} to  find  a Riesz set $E\subseteq \Lambda$ such that $A_E(G)$ is not reflexive. Let $\gamma \in \Gamma\setminus \Lambda$. Then $F_B=A_{\gamma \Lambda}(\Gamma)$ is sAir and $F_1=A_E(G)$ is Arens regular but not reflexive. Since $F_B F_1=\{0\}$, we can apply Corollary \ref{cor:sum}.
\end{proof}

\bibliographystyle{amsplain}
%\bibliography{../bibliografias/bibrep}

\end{document}